\documentclass[11pt]{article}

\usepackage{pgf,acadpgf}
\usepackage{booktabs}
\usepackage{rotating}
\usepackage{amssymb}
\usepackage{amsmath}
\usepackage{theorem}
\usepackage{tikz-cd}
\usepackage[colorlinks=true,linkcolor=blue,citecolor=blue]{hyperref}

\usepackage{multirow}

\usepackage{circledsteps}
\pgfkeys{/csteps/inner xsep=1.5pt}
\pgfkeys{/csteps/inner ysep=1.5pt}
\newcommand\ostar{\mathbin{\Circled{\star}}}

\numberwithin{equation}{section}

\theorembodyfont{\slshape}

  \newtheorem{THM}{Theorem}[section]
  \newtheorem{LEM}[THM]{Lemma}
  \newtheorem{PROP}[THM]{Proposition}
  \newtheorem{COR}[THM]{Corollary}

\theorembodyfont{\rmfamily}

  \newtheorem{EX}[THM]{Example}
  
  \newtheorem{CONV}[THM]{Convention}

\newif\ifQEDsign
\newcommand{\QED}{\global\QEDsigntrue\hfill$\square$}

\newenvironment{proof}%
    {\par\noindent\textit{Proof.}\global\QEDsignfalse}%
    {\ifQEDsign\else\QED\fi\par\bigskip\par}

\def\labelenumi{(\roman{enumi})}

\renewcommand{\le}{\leqslant}
\renewcommand{\ge}{\geqslant}

\newcommand{\0}{\varnothing}

\renewcommand{\phi}{\varphi}
\renewcommand{\epsilon}{\varepsilon}
\newcommand{\UNION}{\bigcup}
\newcommand{\AAA}{\mathbf{A}}

\newcommand{\CC}{\mathbf{C}}

\newcommand{\DD}{\mathbf{D}}

\newcommand{\LO}{\mathbf{Lo}}
\newcommand{\KK}{\mathbf{K}}
\newcommand{\MM}{\mathbf{M}}
\newcommand{\NN}{\mathbb{N}}

\newcommand{\QQ}{\mathbb{Q}}
\newcommand{\RR}{\mathbb{R}}
\renewcommand{\SS}{\mathbf{S}}

\newcommand{\union}{\cup}
\newcommand{\reduct}[2]{\hbox{$#1$}\hbox{$|$}_{#2}}
\newcommand{\restr}[2]{\hbox{$#1$}\hbox{$\upharpoonright$}_{#2}}

\newcommand{\Boxed}[1]{\mbox{$#1$}}

\newcommand{\id}{\mathrm{id}}
\newcommand{\im}{\mathrm{im}}

\newcommand{\ID}{\mathrm{ID}}
\newcommand{\Ob}{\mathrm{Ob}}
\newcommand{\Mor}{\mathrm{Mor}}

\newcommand{\Emb}{\mathrm{Emb}}
\newcommand{\RSurj}{\mathrm{RSurj}}

\newcommand{\op}{\mathrm{op}}
\newcommand{\fin}{\mathit{fin}}
\newcommand{\Fin}{\mathrm{Fin}}

\newcommand{\Age}{\mathrm{Age}}

\newcommand{\iso}{\mathrm{iso}}

\newcommand{\frakA}{\mathfrak{Aut}}
\newcommand{\frakG}{\mathfrak{G}}
\newcommand{\frakH}{\mathfrak{H}}
\newcommand{\frakI}{\mathfrak{Id}}

\newcommand{\subobj}[3]{\binom{#3}{#2}_{#1}}

\newcommand{\calA}{\mathcal{A}}
\newcommand{\calB}{\mathcal{B}}
\newcommand{\calC}{\mathcal{C}}
\newcommand{\calD}{\mathcal{D}}
\newcommand{\calE}{\mathcal{E}}
\newcommand{\calF}{\mathcal{F}}

\newcommand{\calH}{\mathcal{H}}

\newcommand{\calM}{\mathcal{M}}

\newcommand{\calO}{\mathcal{O}}
\newcommand{\calP}{\mathcal{P}}

\newcommand{\calR}{\mathcal{R}}
\newcommand{\calS}{\mathcal{S}}
\newcommand{\calT}{\mathcal{T}}

\newcommand{\Fraisse}{Fra\"\i ss\'e}

\newcommand{\Aut}{\mathrm{Aut}}
\newcommand{\dom}{\mathrm{dom}}
\newcommand{\cod}{\mathrm{cod}}

\newcommand{\Top}{\mathbf{Top}}
\newcommand{\Haus}{\mathbf{Haus}}
\newcommand{\ChEmb}{\mathbf{Lo}_{\mathit{emb}}}
\newcommand{\WchRs}{\mathbf{Wo}_{\mathit{rs}}}

\newcommand{\Borel}{\mathrm{Borel}}
\newcommand{\Borelarrow}{\longrightarrow_\flat}

\title{On Ramsey degrees, compactness and approximability}
\author{Dragan Ma\v sulovi\'c\\
        University of Novi Sad, Faculty of Sciences\\
        Department of Mathematics and Informatics\\
        Trg Dositeja Obradovi\'ca 3, 21000 Novi Sad, Serbia\\
        email: dragan.masulovic@dmi.uns.ac.rs
}

\begin{document}
\maketitle

\begin{abstract}
  One of the consequences of the Compactness Principle in structural Ramsey theory
  is that the small Ramsey degrees cannot exceed the corresponding
  big Ramsey degrees, thereby justifying the choice of adjectives. However, it is unclear what happens in the
  realm of dual Ramsey degrees due to the lack of the compactness argument that applies to that setting.

  In this paper we present a framework within which
  both ``direct'' and dual Ramsey statements can be stated and reasoned about in a uniform fashion.
  We introduce the notion of \emph{approximability} which yields a general compactness argument powerful enough to
  prove statements about both ``direct'' and dual Ramsey phenomena.
  We conclude the paper with an application of the new strategies by generalizing Voigt's $\star$-version of
  the Infinite Ramsey Theorem to a large class of relational structures and deriving a Ramsey
  statement for ``loose colorings'' of enumerated \Fraisse\ limits.

  \medskip

  \noindent
  \textbf{Key Words and Phrases:} Ramsey degrees, Compactness Principle in Ramsey theory, enriched categories

  \medskip

  \noindent
  \textbf{MSC (2020): 05C55; 18D20} 
\end{abstract}

\section{Introduction}

One of the main benefits of exploring structural Ramsey theory using the language of category theory is the
possibility of automatic dualization. For example, a category with the embedding Ramsey property
must have amalgamation~\cite{Nesetril-Rodl} and all of its objects are rigid.
A purely categorical proof of this fact immediately yields a dual result:
a category with the dual embedding Ramsey property must have projective amalgamation, and all of its objects are also rigid.
Unfortunately, this does not apply to the Compactness Principle (see Theorem~\ref{akpt.lem->thm.ramseyF-part12}):
not only its proof involves non-categorical arguments, but also infinite Ramsey phenomena in the dual context require
additional, usually topological restrictions on the colorings. For example, the Infinite Ramsey Theorem holds under no additional assumptions,
while the Infinite Dual Ramsey Theorem holds only for well-structured (e.g.~Borel or Baire) colorings.

One of the consequences of the Compactness Principle is that the small Ramsey degrees cannot exceed the corresponding
big Ramsey degrees, thereby justifying the choice of adjectives. However, it is unclear what happens in the
realm of dual Ramsey degrees due to the lack of the compactness argument that applies to that setting.
Thus, we find ourselves in a situation where we still do not have a compactness argument which is strong enough to show \emph{at the same time} that
the Infinite Ramsey Theorem implies the Finite Ramsey Theorem, and that the Infinite Dual Ramsey Theorem implies the Finite Dual
Ramsey Theorem.

\begin{table}
  \centering
  \begin{tabular}{cccc}
    \cmidrule[\heavyrulewidth]{1-3}
    Ramsey  & \multirow{2}{*}{structural} & \multirow{2}{*}{embedding} \\
    degrees &  &  \\
    \cmidrule{1-3}
    \multirow{2}{*}{``direct''} & small & small \\
    \cmidrule(lr){2-2} \cmidrule(l){3-3}
                           & big  & big \\
    \cmidrule{1-3}
    \multirow{2}{*}{dual}   & small & small & \\
    \cmidrule(lr){2-2} \cmidrule(l){3-3}
                           & \textit{big}   & \textit{big} & $\longleftarrow  \substack{topology\\required}$ \\
    \cmidrule[\heavyrulewidth]{1-3}
  \end{tabular}
  \caption{The taxonomy of Ramsey degrees}
  \label{gfrd.fig.taxonomy}
\end{table}

In the preceding thirty or so years at least eight kinds of Ramsey degrees have been
identified (see Table~\ref{gfrd.fig.taxonomy}): we have small and big Ramsey degrees which can be
``direct'' or dual, and in each case we can consider structural (where we color substructures) or
embedding degrees (where we color morphisms). Moreover, the results concerning the big dual Ramsey
degrees hold only under certain topological restrictions.
Hands-on experience, on the other hand, suggests that all this diversity is superficial and that
foundational distinctions exist only between small and big Ramsey degrees.

In this paper we present a framework within which both ``direct'' and dual Ramsey statements can be stated and reasoned
about in a uniform fashion. All small degrees can be defined in one go, with respect to a category $\CC$ together with a
class of distinguished automorphism groups $(G_A)_{A \in \Ob(\CC)}$ where $G_A \le \Aut_\CC(A)$.
Analogously, all big degrees can be defined in one go, with respect to a category $\CC$ enriched over the
category $\Top$ of topological spaces, together with a
class of distinguished automorphism groups $(G_A)_{A \in \Ob(\CC)}$ where $G_A \le \Aut_\CC(A)$.
All the variations given in Table~\ref{gfrd.fig.taxonomy} can then be obtained by specifying
a category, an enrichment and a family of automorphism groups.

In Section~\ref{gfrd.sec.prelim} we fix notation and list some basic facts about relational structures,
\Fraisse\ theory, structural Ramsey theory and category theory.
In Section~\ref{crt.sec.general-def} we introduce a general point of view of small Ramsey degrees and prove
a generalization of Zucker's result from \cite{zucker1} that correlates the embedding and non-embedding small Ramsey degrees
(see Proposition~\ref{gfrd.prop.small-degs}).
In Section\ref{gfrd.sec.big-degs} we present a general point of view on big Ramsey degrees in categories enriched over $\Top$
using the fact that locally compact second-countable Hausdorff enrichment is stable under factoring by compact topological groups.
Moreover, this enrichment is compatible with the two typical situations we would like to model. ``Direct'' big Ramsey degrees are usually computed
in categories with discrete enrichment and with respect to all colorings, while dual big Ramsey degrees are usually computed
in categories where the homsets are endowed with the pointwise convergence topology and with respect to Borel colorings.
Both of these contexts fit into the framework of Borel colorings with respect to a locally compact second-countable Hausdorff enrichment.
In Section~\ref{gfrd.sec.compactness} we generalize the standard compactness arguments and introduce the notion of \emph{approximability}
which yields a general compactness argument (Theorem~\ref{gfrd.thm.COMPACTNESS-0}). We conclude the paper with Section~\ref{gfrd.sec.star-ramsey}
in which we present an application of the strategies developed in the preceding sections. We generalize Voigt's
$\star$-version of the Infinite Ramsey theorem to a large class of relational structures (Theorem~\ref{gfrd.thm.struct-enum-star-prop}) and derive a Ramsey
statement for ``loose colorings'' of enumerated \Fraisse\ limits (Theorem~\ref{gfrd.thm.struct-star-property}).

\section{Preliminaries}
\label{gfrd.sec.prelim}

\paragraph{Rigid surjections and parameter words.}
If $(A, \Boxed<)$ and $(B, \Boxed<)$ are well-ordered sets,
a surjection $f : A \to B$ is \emph{rigid} if $\min f^{-1}(b) < \min f^{-1}(b')$ whenever $b < b'$ in~$B$.
Let $\RSurj(A, B)$ denote the set of all rigid surjections $A \to B$.

Let $A = \{a_1 < \ldots < a_k\}$ be a finite linearly ordered alphabet. A word $u$ of length $n$ over $A$ can be thought of as
an element of $A^n$ but also as a mapping $u : \{1, 2, \ldots, n\} \to A$. In the latter case
$u^{-1}(a)$, $a \in A$, denotes the set of all the positions in $u$ where $a$ appears.
Let $X = \{x_1 < x_2 < \ldots\}$ be a countably infinite linearly ordered set of variables disjoint from $A$.
For $m \ge 1$, an \emph{$m$-parameter word over $A$ of length $n$} is a word
$w \in (A \union \{x_1, x_2, \ldots, x_m\})^n$ such that
  each of the letters $x_1, \ldots, x_m$ appears at least once in $w$, and
  $\min(w^{-1}(x_i)) < \min(w^{-1}(x_j))$ whenever $1 \le i < j \le m$.
These notions extend straightforwardly to the case where $n = \omega$ and $m$ is finite, of $n = m = \omega$.
For example, an \emph{$\omega$-parameter word over $A$ of length $\omega$} is a word
$w \in (A \union X)^\omega$ such that each letter from $X$ appears at least once in $w$, and
$\min(w^{-1}(x_i)) < \min(w^{-1}(x_j))$ whenever $i < j$.

Let $W^n_m(A)$ denote the set of all the $m$-parameter words over $A$ of length~$n$, $1 \le m, n \le \omega$.
For $u \in W^n_m(A)$ and $v = v_1 v_2 \ldots \in W^m_k(A)$ let
$$
  u \cdot v = u[v_1/x_1, v_2/x_2, \ldots] \in W^n_k(A)
$$
denote the word obtained by replacing each occurrence of $x_i$ in $u$ with $v_i$,
simultaneously for all $1 \le i \le m$.

It will also be convenient to consider partial substitutions of parameter words defined as follows.
For $u \in W^n_m(A)$ and $v = v_1 v_2 \ldots \in W^\ell_k(A)$ where $\ell \le m$ let
$u \star v$ denote the word obtained by replacing each occurrence of $x_i$ in $u$ with $v_i$ for $1 \le i \le \ell$,
and truncating the word $u$ at the first occurrence of $x_{\ell+1}$ (if $x_{\ell+1}$ appears in $u$).
The case where $u$ is an $\omega$-parameter word and $v$ is a finite word will be of particular interest.
To compute $u \star v$ we scan $u$ from left to right and substitute variables that appear in $u$ by
the corresponding letters from $v$. Once we reach a variable $x_s$ where $s > |v|$ we stop and cut off the remainder of~$u$.
Thus, if $v$ is a finite word, $u \star v$ is also finite.

Parameter words are clearly related to rigid surjections.
To an $m$-parameter word $u = u_1 n_2 \ldots u_n \in W^n_m(A)$ of length $n$ we assign a rigid surjection
$$
  f_u : \{a_1 < \ldots < a_k < 1 < \ldots < n\} \to \{a_1 < \ldots < a_k < x_1 < \ldots x_m\}
$$
so that $f_u(a_i) = a_i$, $1 \le i \le k$, and $f_u(j) = u_j$, $1 \le j \le n$.
It is easy to see that the substitution of parameter words corresponds precisely to
the composition of rigid surjections.
Partial substitution of parameter words can also be interpreted as a special form of composition of
rigid surjections, see Section~\ref{gfrd.sec.star-ramsey}.

\paragraph{Relational structures.}
Let $L$ be a relational language. An \emph{$L$-structure} is a set $A$ together with a family of finitary relations on $A$
interpreting the symbols from $L$: $\calA = (A, \{R^A : R \in L\})$. We shall often write $\calA = (A, L^A)$.
If $L_0 \subseteq L$ then the \emph{$L_0$-reduct} of an $L$-structure $\calA = (A, L^A)$
is the $L_0$ structure $\reduct \calA {L_0} = (A, \{R^A : R \in L_0\}) = (A, L_0^A)$.

Let $L$ be a relational language and let $\calA = (A, L^A)$ and $\calB = (B, L^B)$ be $L$-structures.
An \emph{embedding} $h : \calA \hookrightarrow \calB$ is an injective map $A \to B$ such that
$R^B(h(a_1), \ldots, h(a_n)) \Leftrightarrow R^A(a_1, \ldots, a_n)$ for all $R \in L$ and $a_1, \ldots, a_n \in A$,
where $n$ is the arity of~$R$. Let $\Emb(\calA, \calB)$ denote the set of all the embeddings $\calA \hookrightarrow \calB$.
An \emph{isomorphism} is a bijective embedding. We write $\calA \cong \calB$ to denote that $\calA$ and $\calB$ are
isomorphic.
We write $\calA \le \calB$ to denote that $\calA$ is a \emph{substructure} of $\calB$, that is,
$A \subseteq B$ and $R^A = R^B \cap A^n$ for every $R \in L$ (where $n$ is the arity of $R$).
Every nonempty $C \subseteq B$ uniquely determines a \emph{substructure of $\calB$ induced by $C$},
in symbols $\calB[C]$, which is the $L$-structure $\calC = (C, L^C)$ such that
$R^C = R^B \cap A^n$ for all $R \in L$ (where $n$ is the arity of $R$). For $L$-structures $\calA, \calB$ let
$$
  \binom \calB \calA = \{\calE : \calE \cong \calA \land \calE \le \cal B\}.
$$
If $h : \calA \hookrightarrow \calB$ is an embedding of $L$-structures, then $\im(h)$, \emph{the image of $h$},
denotes the substructure of $\calB$ induced by $\{h(a) : a \in A\}$.

Let $\LO$ denote the class of all linearly ordered sets $(A, \Boxed<)$. Then
for a class $\KK$ of $L$-structures we let
$$
  \KK * \LO = \{(A, L^A, \Boxed<) : (A, L^A) \in \KK \text{ and } (A, \Boxed<) \in \LO \}.
$$
Note that $\KK * \LO$ is a class of $L'$-structures where $L' = L \union \{\Boxed<\}$.

\paragraph{\Fraisse\ theory.}
The \emph{age} of a countably infinite $L$-structure $\calF$ is the class $\Age(\calF)$ of all
finite structures that embed into~$\calF$. A class $\KK$ of finite $L$-structures is an \emph{age}
if there is countably infinite $L$-structure $\calF$ such that
$\KK = \Age(\calF)$. It is easy to see that $\KK$ is an age if and only if:
  $\KK$ is closed under taking isomorphic copies;
  there are at most countably many pairwise nonisomorphic structures in $\KK$;
  $\KK$ has the \emph{hereditary property (HP) with respect to $\AAA$}:
  if $\calB \in \KK$ and $\calA \in \AAA$ such that $\calA \hookrightarrow \calB$ then $\calA \in \KK$;
  and
  $\KK$ is \emph{directed}:
  for all $\calA, \calB \in \KK$ there is a $\calC \in \KK$ such that
  $\calA \hookrightarrow \calC$ and $\calB \hookrightarrow \calC$.

An age $\KK$ is a \emph{\Fraisse\ age} if $\KK$ has \emph{amalgamation}:
for all $\calA, \calB, \calC \in \KK$ and embeddings $f : \calA \hookrightarrow \calB$ and
$g : \calA \hookrightarrow \calC$ there exist $\calD \in \KK$ and embeddings $f' : \calB \hookrightarrow \calD$ and
$g' : \calC \hookrightarrow \calD$ such that $f' \circ f = g' \circ g$.

An $L$-structure $\calF$ is \emph{ultrahomogeneous} if for every $\calA \in \Age(\calF)$ and every pair of
embeddings $f, g : \calA \hookrightarrow \calF$ there is an automorphism $h \in \Aut(\calF)$ such that
$f = h \circ g$. The age of every countably infinite ultrahomogeneous structure is a \Fraisse\ age.
Conversely, for every \Fraisse\ age $\KK$ there is a unique (up to isomorphism)
countably infinite ultrahomogeneous $L$-structure $\calF$ such that $\KK = \Age(\calF)$~\cite{Fraisse1,Fraisse2,hodges}.
We say that $\calF$ is the \emph{\Fraisse\ limit of $\KK$.}

A class $\KK$ of finite $L$-structures has the \emph{strong amalgamation property} if 
for all $\calA, \calB, \calC \in \KK$ and embeddings $f : \calA \hookrightarrow \calB$ and
$g : \calA \hookrightarrow \calC$ there exist $\calD \in \KK$ and embeddings $f' : \calB \hookrightarrow \calD$ and
$g' : \calC \hookrightarrow \calD$ such that $f' \circ f = g' \circ g$ and $f'(\calB) \cap g'(\calC) = f'(f(\calA))$.
A \Fraisse\ limit $\calF$ has \emph{strongly amalgamable age} if $\Age(\calF)$ has the strong amalgamation property.
Here is an example of some well-known \Fraisse\ classes and their \Fraisse\ limits.
They all have strongly amalgamable ages.

\begin{EX}\label{bigrd.ex.fraisse-classes}
  \begin{itemize}
  \item
    The class of all finite linear orders is a \Fraisse\ age and its \Fraisse\ limit is the usual order of the rationals $\QQ$~\cite{Fraisse1,Fraisse2}.
  \item
    The class of all finite graphs is a \Fraisse\ age and its \Fraisse\ limit is the \emph{random graph} $\calR$~\cite{Erdos-Renyi}.
  \item
    For each $n \ge 3$ the class of all finite $K_n$-free graphs is a \Fraisse\ age and its \Fraisse\ limit is the \emph{Henson graph} $\calH_n$~\cite{henson}.
  \item
    The class of all finite partially ordered sets is a \Fraisse\ age and its \Fraisse\ limit is the \emph{random poset}~$\calP$~\cite{Schmerl}.
  \item
    The class of all finite digraphs is a \Fraisse\ age and its \Fraisse\ limit is the \emph{random digraph} $\calD$.
  \item
    The class of all finite tournaments is a \Fraisse\ age and its \Fraisse\ limit is the \emph{random tournament} $\calT$.
  \item
    Let $\MM(\Delta)$ be the class of all finite metric spaces whose distances are in $\Delta \subseteq \RR$.
    A detailed analysis of those sets $\Delta$ of nonnegative reals for which $\MM(\Delta)$ is a \Fraisse\ age can
    be found in~\cite{dlps-2007} and~\cite{Sauer-2013}). If $\MM(\Delta)$ is a \Fraisse\ age its \Fraisse\ limit will be referred to as the
    \emph{random $\Delta$-metric space} and denoted by~$\calM(\Delta)$.
  \end{itemize}
\end{EX}

\paragraph{Structural Ramsey Theory.}
The leitmotif of Ramsey theory is to prove the existence of regular patterns that occur when a large structure is considered
in a combinatorially restricted context. One of its most prominent early results, which actually gave the name to the theory, is
the Infinite Ramsey Theorem. Spelled in terms of linear orders, where $\omega = \{0, 1, 2, \ldots\}$ and $n = \{0, 1, \ldots, n-1\}$,
it reads:

\begin{THM}[The Infinite Ramsey Theorem~\cite{Ramsey}]\label{gfrd.thm.IRT}
  For every finite linear order $n \in \NN$, every $k \in \NN$ and every coloring $\chi : \Emb(n, \omega) \to k$
  there is a $w \in \Emb(\omega, \omega)$ such that $|w \circ \Emb(n, \omega)| = 1$.
\end{THM}

\noindent
Its finite version is another famous result:

\begin{THM}[The Finite Ramsey Theorem~\cite{Ramsey}]\label{gfrd.thm.FRT}
  For all $\ell, m \in \NN$ and the number of colors $k \in \NN$ there is an $n \in \NN$ such that
  for every coloring $\chi : \Emb(\ell, n) \to k$ there is a $w \in \Emb(m, n)$ such that
  $|\chi(w \circ \Emb(\ell, m))| = 1$.
\end{THM}

\noindent
Finite linear orders and rigid surjections exhibit another Ramsey-related phenomenon:

\begin{THM}[The Finite Dual Ramsey Theorem~\cite{GR}]\label{gfrd.thm.FDRT}
  For all $\ell, m \in \NN$ and the number of colors $k \in \NN$ there is an $n \in \NN$ such that
  for every coloring $\chi : \RSurj(n, \ell) \to k$ there is a $w \in \RSurj(n, m)$ such that
  $|\chi(\RSurj(m, \ell) \circ w)| = 1$.
\end{THM}

Finite Dual Ramsey Theorem has an infinite version due to Carlson and Simpson~\cite{carlson-simpson-1984}.
However, the formulation of the Infinite Dual Ramsey Theorem relies on some topological infrastructure.
For a pair of well-ordered sets $(A, \Boxed<)$ and $(B, \Boxed<)$ the set $\RSurj(A, B)$ is endowed
with the pointwise convergence topology, and with respect to that topology we say that
a coloring $\chi : \RSurj(A, B) \to k$ is \emph{Borel} if $\chi^{-1}(j)$ is a Borel set for all $j < k$.

\begin{THM}[The Infinite Dual Ramsey Theorem~\cite{carlson-simpson-1984}]\label{gfrd.thm.IDRT}
  For every $n \in \NN$, every $k \in \NN$ and every Borel coloring $\chi : \RSurj(\omega, n) \to k$
  there is a $w \in \RSurj(\omega, \omega)$ such that $|\chi(\RSurj(\omega, n) \circ w)| = 1$.
\end{THM}

Generalizing the Finite Ramsey Theorem from linear orders to arbitrary relational structures, the structural
Ramsey theory originated at the beginning of 1970’s in a series of papers (see \cite{N1995} for references). 
Many natural classes of structures such as finite graphs and finite posets
do not have the Ramsey property. Nevertheless, many of these classes enjoy
the weaker property of having finite small Ramsey degrees. Over the years several kinds of Ramsey degrees
have been identified.

Let $L$ be a relational language and $\KK$ a class of finite $L$-structures.
A \emph{small embedding Ramsey degree of $\calA \in \KK$} is the least positive integer $n \in \NN$, if such an integer exists,
with the property that for every $\calB  \in \KK$ and every $k \in \NN$ there is a $\calC \in \KK$ such that for every coloring
$\chi : \Emb(\calA, \calC) \to k$ one can find a $w \in \Emb(\calB, \calC)$ satisfying $|\chi(w \circ \Emb(\calA, \calB))| \le n$.
We then write $t(\calA) = n$. If no such $n \in \NN$ exists we write $t(\calA) = \infty$. A class $\KK$ of finite $L$-structures
\emph{has finite small embedding Ramsey degrees} if $t(\calA) < \infty$ for all $\calA \in \KK$, and it \emph{has the embedding Ramsey property} if
$t(\calA) = 1$ for all $\calA \in \KK$.

A \emph{small structural Ramsey degree of $\calA \in \KK$} is the least positive integer $n \in \NN$, if such an integer exists,
with the property that for every $\calB  \in \KK$ and every $k \in \NN$ there is a $\calC \in \KK$ such that for every coloring
$\chi : \subobj{}\calA\calC \to k$ one can find a $w \in \Emb(\calB, \calC)$ satisfying $|\chi(w \circ \subobj{} \calA \calB)| \le n$.
We then write $\tilde t(\calA) = n$. If no such $n \in \NN$ exists we write $\tilde t(\calA) = \infty$. A class $\KK$ of finite $L$-structures
\emph{has finite small structural Ramsey degrees} if $\tilde t(\calA) < \infty$ for all $\calA \in \KK$, and it \emph{has the structural Ramsey property} if
$\tilde t(\calA) = 1$ for all $\calA \in \KK$.

It was shown in~\cite{zucker1} that the two kinds of small Ramsey degrees are closely related: for a finite relational structure $\calA$
coming from a class $\KK$ of finite relational structures we have that $t(\calA) = |\Aut(\calA)| \cdot \tilde t(\calA)$
(see~\cite{masul-kpt} for a slight generalization).

Big Ramsey degrees were first introduced in the context of structural Ramsey theory in~\cite{KPT}.
Let $\calA$ and $\calS$ be $L$-structures such that $\calA$ is finite.
A \emph{big embedding Ramsey degree of $\calA$  in $\calS$} is the least positive integer $n \in \NN$, if such an integer exists,
with the property that for every $k \in \NN$ and every coloring
$\chi : \Emb(\calA, \calS) \to k$ one can find a $w \in \Emb(\calS, \calS)$ satisfying $|\chi(w \circ \Emb(\calA, \calS))| \le n$.
We then write $T(\calA, \calS) = n$. If no such $n \in \NN$ exists we write $T(\calA, \calS) = \infty$. We say that
\emph{$\calS$ has finite big embedding Ramsey degrees} if $T(\calA, \calS) < \infty$ for all $\calA \in \Age(\calS)$.

A \emph{big structural Ramsey degree of $\calA$  in $\calS$} is the least positive integer $n \in \NN$, if such an integer exists,
with the property that for every $k \in \NN$ and every coloring
$\chi : \subobj{} \calA \calS \to k$ one can find a $w \in \Emb(\calS, \calS)$ satisfying $|\chi(w \circ \subobj{} \calA \calS)| \le n$.
We then write $\tilde T(\calA, \calS) = n$. If no such $n \in \NN$ exists we write $\tilde T(\calA, \calS) = \infty$. We say that
\emph{$\calS$ has finite big structural Ramsey degrees} if $\tilde T(\calA, \calS) < \infty$ for all $\calA \in \Age(\calS)$.

It comes as no surprise that the two kinds of big Ramsey degrees are also closely related:
it was shown in~\cite{Zucker-2} that $T(\calA, \calS) = |\Aut(\calA)| \cdot \tilde T(\calA, \calS)$
for a finite relational structure $\calA$ that embeds into a relational structure $\calS$ (see~\cite{masul-kpt} for a slight generalization).

\begin{EX}
  $(a)$
  The order of the rationals, $\QQ$, has finite big Ramsey degrees~\cite{galvin1,galvin2}.
  The exact values of $T(n, \QQ)$, where $n \in \NN$, were computed in~\cite{devlin}.

  $(b)$
  It follows from Sauer's results in \cite{Sauer-2006} that the Rado graph $\calR$, the random tournament $\calT$ and the random digraph $\calD$ have finite big Ramsey degrees.

  $(c)$
  Dobrinen proved in \cite{dobrinen2} that the Henson graphs $\calH_n$ have finite big Ramsey degrees for every $n \ge 3$.

  $(d)$
  Hubi\v cka proved in \cite{hubicka-param-spaces} that the random poset $\calP$ has finite big Ramsey degrees,
  and inferred from that using a construction from~\cite{masul-preadj}
  that the random metric space $\calM(\Delta)$ has finite big Ramsey degrees
  for certain choices of $\Delta \subseteq \RR$ (see \cite{hubicka-param-spaces} for details).
  In particular, $\calM(S)$ has finite big Ramsey degrees for every finite $S \subseteq \RR$ for which $\calM(S)$ exists~\cite{hubicka-forb-cyc}.
  Interestingly, neither $\calM(\QQ)$ nor $\calM(\NN)$ have finite big Ramsey degrees~\cite{hubicka-personal}.

  $(e)$
  Zucker proved in \cite{zucker-brd-bin-free-amal-classes} that given a finite binary relational language and a
  finite set of ``forbidden'' finite irreducible structures, the \Fraisse\ limit of the
  class of all finite structures that embed none of the forbidden finite structures has finite big Ramsey degrees.
\end{EX}

\begin{CONV}\label{crt.CONVENTION.Ninf}
  Let $\NN_\infty = \NN \union \{\infty\} = \{1, 2, 3, \ldots, \infty \}$.
  The usual linear order on the positive integers extends to $\NN_\infty$ straightforwardly:
  $
    1 < 2 < \ldots < \infty
  $.
    Addition and multiplication also extend to $\NN_\infty$ straightforwardly:
  we take $\infty + n = n + \infty = \infty + \infty = \infty$ and
  $\infty \cdot n = n \cdot \infty = \infty \cdot \infty = \infty$.
  Ramsey degrees take their values in $\NN_\infty$. Therefore, if $t$ is a Ramsey degree
  and $A$ is a finite or countably infinite set, by $t \ge |A|$ we mean the following:
    $t \in \NN$, $|A| \in \NN$ and $t \ge |A|$; or
    $t = \infty$ and $|A| \in \NN$; or
    $A$ is a countably infinite set and $t = \infty$.
  If $A$ and $B$ are sets then $|A| \ge |B|$ has the usual meaning.
\end{CONV}

\paragraph{Categories.}
In order to specify a \emph{category} $\CC$ one has to specify
a class of objects $\Ob(\CC)$, a class of morphisms $\Mor(\CC)$,
functions $\dom, \cod : \Mor(\CC) \to \Ob(\CC)$,
the identity morphism $\id_A$ for all $A \in \Ob(\CC)$ such that $\dom(\id_A) = \cod(\id_A) = A$, and
the composition of morphisms~$\cdot$~so that
$\id_B \cdot f = f = f \cdot \id_A$ for all $f \in \Mor(\CC)$ with $\dom(f) = A$ and $\cod(f) = B$,
and $(f \cdot g) \cdot h = f \cdot (g \cdot h)$ whenever $\cod(h) = \dom(g)$ and $\cod(g) = \dom(f)$.
For $A, B \in \Ob(\CC)$ by $\hom_\CC(A, B)$ we denote the class of all $f \in \Mor(\CC)$ such that $\dom(f) = A$ and $\cod(f) = B$.
As usual, we assume that $\hom_\CC(A, B)$ and $\hom_\CC(C, D)$ are disjoint whenever $(A, B) \ne (C, D)$.
A category $\CC$ is \emph{locally small} if $\hom_\CC(A, B)$ is a set for all $A, B \in \Ob(\CC)$.
Sets of the form $\hom_\CC(A, B)$ are then referred to as \emph{hom-sets}.
Write $A \to B$ if $\hom_\CC(A, B) \ne \0$. 
A locally small category $\CC$ is \emph{small} if $\Ob(\CC)$ is a set.

\begin{EX}
  Every class of first-order structures can be understood as a locally small category whose morphisms are embeddings
  of first-order structures. This is the intended interpretation whenever a class of first-order structures is treated as a category
  and the morphisms are not specified. In particular, let $\ChEmb$ denote the category whose objects are all linear orders
  and morphisms are embeddings.
\end{EX}

\begin{EX}
  By $\Top$ we denote the category of all topological spaces and continuous maps between them.
\end{EX}

\begin{EX}\label{p-a-t.ex.3}
  Let $\WchRs$ denote the category whose objects are well-ordered sets and morphisms are rigid surjections.
\end{EX}

A category $\CC$ is \emph{directed} if for every $A, B \in \Ob(\CC)$ there is a $C \in \Ob(\CC)$ such that $A \to C$ and $B \to C$,
and it \emph{has amalgamation} if for all $A, B_1, B_2 \in \Ob(\CC)$ and morphisms $f_1 \in \hom_\CC(A, B_1)$, $f_2 \in \hom_\CC(A, B_2)$
there is a $C \in \Ob(\CC)$ and morphisms $g_1 \in \hom_\CC(B_1, C)$, $g_2 \in \hom_\CC(B_2, C)$ such that $g_1 \cdot f_1 = g_2 \cdot f_2$.

A morphism $f$ is: \emph{mono} or \emph{left cancellable} if
$f \cdot g = f \cdot h$ implies $g = h$ whenever the compositions make sense;
\emph{epi} or \emph{right cancellable} if
$g \cdot f = h \cdot f$ implies $g = h$ whenever the compositions make sense; and
\emph{invertible} if there is a morphism $g$ with the appropriate domain and codomain
such that $g \cdot f = \id$ and $f \cdot g = \id$.
By $\iso_\CC(A, B)$ we denote the set of all invertible
morphisms $A \to B$, and we write $A \cong B$ if $\iso_\CC(A, B) \ne \0$. Let $\Aut_\CC(A) = \iso_\CC(A, A)$.
An object $A \in \Ob(\CC)$ is \emph{rigid} if $\Aut_\CC(A) = \{\id_A\}$.

Given a category $\CC$, the \emph{opposite category} $\CC^\op$ is a category constructed from $\CC$ on the same class of objects
by formally reversing arrows and composition. More precisely, for $A, B \in \Ob(\CC) = \Ob(\CC^\op)$ we have that
$\hom_{\CC^\op}(A, B) = \hom_{\CC}(B, A)$, and for $f \in \hom_{\CC^\op}(A, B)$ and $g \in \hom_{\CC^\op}(B, C)$
we have that $g \mathbin{{\cdot}_{\CC^\op}} f = f \mathbin{{\cdot}_{\CC}} g$.

A category $\DD$ is a \emph{subcategory} of a category $\CC$ if $\Ob(\DD) \subseteq \Ob(\CC)$ and
$\hom_\DD(A, B) \subseteq \hom_\CC(A, B)$ for all $A, B \in \Ob(\DD)$.
A category $\DD$ is a \emph{full subcategory} of a category $\CC$ if $\Ob(\DD) \subseteq \Ob(\CC)$ and
$\hom_\DD(A, B) = \hom_\CC(A, B)$ for all $A, B \in \Ob(\DD)$. If $\DD$ is a full subcategory of $\CC$ we also
say that $\DD$ is \emph{a subcategory of $\CC$ spanned by} $\Ob(\DD)$.

\begin{EX}
  Let $\ChEmb^\fin$ denote the full subcategory of $\ChEmb$ spanned by all finite linear orders;
  let $\WchRs^\fin$ denote the full subcategory of $\WchRs$ spanned by all finite linear orders;
  let $\Haus$ denote the full subcategory of $\Top$ spanned by Hausdorff spaces.
\end{EX}

A \emph{skeleton of $\CC$} is a full subcategory $\SS$ of $\CC$ such that every object of $\CC$ is isomorphic
to some object in $\SS$, and no two objects of $\SS$ are isomorphic. In other words,
$\SS$ contains exactly one representative of each isomorphism class of objects in~$\CC$.
A category $\CC$ is skeletal if it equals its skeleton, that is, if there are no distinct isomorphic objects in~$\CC$.

We shall say that a category $\CC$ is a \emph{category of finite objects} if
$\CC$ is a locally small directed category whose morphisms are mono,
the skeleton $\SS$ of $\CC$ has at most countably many objects, and
for every $S \in \Ob(\SS)$ there are only finitely many morphisms in $\SS$ whose codomain is $S$.
In particular, in a category of finite objects every homset is a finite set of monos.

A \emph{functor} $F : \CC \to \DD$ from a category $\CC$ to a category $\DD$ maps $\Ob(\CC)$ to
$\Ob(\DD)$ and maps morphisms of $\CC$ to morphisms of $\DD$ so that
$F(f) \in \hom_\DD(F(A), F(B))$ whenever $f \in \hom_\CC(A, B)$, $F(f \cdot g) = F(f) \cdot F(g)$ whenever
$f \cdot g$ is defined, and $F(\id_A) = \id_{F(A)}$.
A functor $F : \CC \to \CC$ such that $F(A) = A$ and $F(f) = f$ for all objects $A$ and morphisms $f$
is called the \emph{identity functor} and denoted by~$\ID_\CC$.
Categories $\CC$ and $\DD$ are \emph{isomorphic}, in symbols $\CC \cong \DD$, if there exist
functors $F : \CC \to \DD$ and $G : \DD \to \CC$ such that $G \circ F = \ID_\CC$ and $F \circ G = \ID_\DD$.


A functor $F : \CC \to \DD$ is \emph{full} if it is surjective on homsets (that is: for
every $g \in \hom_\DD(F(A), F(B))$ there is an $f \in \hom_\CC(A, B)$ with $F(f) = g$),
and \emph{faithful} if it is injective on homsets (that is: $F(f) = F(g)$ implies $f = g$).
A functor $F : \CC \to \DD$ is \emph{isomorphism-dense} if for every $D \in \Ob(\DD)$ there is
a $C \in \Ob(\CC)$ such that $F(C) \cong D$. A functor $F : \CC \to \DD$ is an \emph{equivalence}
if it is full, faithful and isomorphism-dense. Categories $\CC$ and $\DD$ are \emph{equivalent}
if there is an equivalence $F : \CC \to \DD$.

\section{Small Ramsey degrees}
\label{crt.sec.general-def}

We shall now present another, more general point of view of small Ramsey degrees
whose special cases are both structural and embedding Ramsey degrees of an object in a category.

Let $\CC$ be a locally small category and let $\frakG = (G_A)_{A \in \Ob(\CC)}$
be a sequence of groups indexed by the objects of the category chosen so that
$G_A \le \Aut(A)$ for all $A \in \Ob(\CC)$. For $f, g \in \hom(A, B)$
write $f \sim_\frakG g$ to denote that there is an $\alpha \in G_A$ such that $f = g \cdot \alpha$.
It is easy to see that $\sim_\frakG$ is an equivalence relation on $\hom(A, B)$, so we let
$$\textstyle
  \subobj\frakG A B = \hom(A, B) / \Boxed{\sim_\frakG} = \{f \cdot G_A : f \in \hom(A, B)\}.
$$
For integers $k, t \in \NN$ and $A, B, C \in \Ob(\CC)$ such that $A \to B \to C$
we write
$$
  C \overset\frakG\longrightarrow (B)^{A}_{k, t}
$$
to denote that for every $k$-coloring
$
  \chi : \subobj\frakG AC \to k
$
there is a morphism $w : B \to C$ such that $|\chi(w \cdot \subobj\frakG AB)| \le t$.
For $A \in \Ob(\CC)$ let $t^\frakG_\CC(A)$, the \emph{small $\frakG$-Ramsey degree of $A$ in $\CC$},
denote the least positive integer $n$ such that
for all $k \in \NN$ and all $B \in \Ob(\CC)$ there exists a $C \in \Ob(\CC)$ such that
$C \overset\frakG\longrightarrow (B)^{A}_{k, n}$, if such an integer exists.
Otherwise put $t^\frakG_\CC(A) = \infty$.

It is easy to see that in the two extreme cases $\frakG$-Ramsey degrees reduce to embedding
Ramsey degrees and structural Ramsey degrees. Namely, let $\frakA = (\Aut(A))_{A \in \Ob(\CC)}$
and $\frakI = (\{\id_A\})_{A \in \Ob(\CC)}$. Then, for all $A \in \Ob(\CC)$,
$t^\frakA_\CC(A)$ is usually referred to as the \emph{small structural Ramsey degree of $A$ in $\CC$}, and
$t^\frakI_\CC(A)$ is usually referred to as the \emph{small embedding Ramsey degree of $A$ in $\CC$}.
This is because $\subobj{\frakA}AB = \hom(A, B) / \Boxed{\sim_{\Aut(A)}}$ and, with a slight abuse of notation, $\subobj{\frakI}AB = \hom(A, B)$.
On the other hand, small dual Ramsey degrees are nothing but ``direct'' Ramsey degrees
in the opposite category. Therefore, the infrastructure provided by a locally small category
together with a sequence of distinguished automorphism groups
$\frakG = (G_A)_{A \in \Ob(\CC)}$ suffices for the uniform treatment of all four kinds
of small Ramsey degrees, Table~\ref{gfrd.fig.taxonomy-small}.
Small embedding Ramsey degrees will be denoted by $t_\CC(A)$, and we let $t^\partial_\CC(A) = t_{\CC^\op}(A)$ denote the small dual
embedding Ramsey degrees. We shall omit the index $\CC$ whenever there is no possibility of confusion.

\begin{table}
  \centering
  \begin{tabular}{cccc}
    \cmidrule[\heavyrulewidth]{1-3}
    Small Ramsey  & \multirow{2}{*}{structural} & \multirow{2}{*}{embedding} \\
    degrees &  &  \\
    \cmidrule{1-3}
    ``direct'' & $t^\frakA_\CC(A)$ & $t^\frakI_\CC(A) = t_\CC(A)$ \\
    \cmidrule{1-3}
    dual   & $t^\frakA_{\CC^\op}(A)$ & $t^\frakI_{\CC^\op}(A) = t^\partial_\CC(A)$ \\
    \cmidrule[\heavyrulewidth]{1-3}
  \end{tabular}
  \caption{Small Ramsey degrees are instances of the same phenomenon}
  \label{gfrd.fig.taxonomy-small}
\end{table}

We shall often implicitly rely on the following simple observation:

\begin{LEM}
  Let $\CC$ be a locally small category and let $\frakG = (G_A)_{A \in \Ob(\CC)}$ be a sequence of groups
  indexed by the objects of the category chosen so that $G_A \le \Aut(A)$ for all $A \in \Ob(\CC)$.
  Let $A \in \Ob(\CC)$  and $t \in \NN$ be arbitrary. Then
  $f \cdot (x / \Boxed{\sim_\frakG}) = (f \cdot x) / \Boxed{\sim_\frakG}$ for all $x \in \hom(A, B)$ and $f \in \hom(B, C)$.
\end{LEM}

We say that a locally small category $\CC$ has the \emph{$\frakG$-Ramsey property} if $t^\frakG_\CC(A) = 1$ for all $A \in \Ob(\CC)$.
The two extreme cases are usually referred to as the \emph{embedding Ramsey property} for $\frakG = \frakI$,
and the \emph{structural Ramsey property} for $\frakG = \frakA$.

\begin{PROP}
  Let $\CC$ be a locally small category whose morphisms are mono.
  If $\CC$ has the $\frakG$-Ramsey property then $\Aut_\CC(A) = G_A$ for all $A \in \Ob(\CC)$.
  In particular, if $\CC$ has the embedding Ramsey property all the objects in $\CC$ are rigid.
\end{PROP}
\begin{proof}
  Assume that there is an $A \in \Ob(\CC)$ such that $\Aut_\CC(A) \ne G_A$ and let us show that
  for every $C \in \Ob(\CC)$ it is \emph{not} the case that $C \overset\frakG\longrightarrow (A)^A_2$,
  showing, thus, that $\CC$ does not have the $\frakG$-Ramsey property.

  Let $C \in \Ob(\CC)$ be any object of $\CC$ such that $\hom(A, C) \ne \0$
  and take any $\alpha \in \Aut_\CC(A) \setminus G_A$.
  The idea is now simple: construct a coloring $\subobj{\frakG}AC \to 2$ which
  colors classes of the form $f \cdot G_A$ by~0, and classes of the form $f \cdot \alpha \cdot G_A$ by~1.
  However, a bit of care is needed to provide a correct definition.
  
  Let $H$ be the subgroup of $\Aut_\CC(A)$ generated by $\{\alpha\} \cup G_A$.
  Let $H$ act on $\hom_\CC(A, C)$ by $f^h = f \cdot h$
  where $h \in H$ and $f \in \hom_\CC(A, C)$. Let $\{\calO_i : i \in I\}$ be the
  set of orbits in this action, and for each orbit $\calO_i$ choose a representative $f_i \in \calO_i$
  so that $\calO_i = f_i \cdot H$, $i \in I$. Each orbit $\calO_i$ now contains both the class
  $f_i \cdot G_A$ and the class $f_i \cdot \alpha \cdot G_A$ and it is easy to check that
  the two classes are disjoint. (This follows from the fact that $f_i$ is mono and that $\alpha \notin G_A$.)
  
  Consider the coloring $\chi : \subobj\frakG AC \to 2$ defined so that $\chi(f_i \cdot G_A) = 0$
  and $\chi(f_i \cdot \alpha \cdot G_A) = 1$ for each $i \in I$. For other classes of $\sim_\frakG$ define
  $\chi$ arbitrarily. Take any $w \in \hom(A, C)$ and let us show that $|\chi(w \cdot \subobj\frakG AA)| = 2$.
  From $\hom(A, A) \supseteq H$ it follows that
  $$\textstyle
    w \cdot \subobj\frakG AA = w \cdot \hom(A, A)/\Boxed{\sim_\frakG} \supseteq w \cdot H/\Boxed{\sim_\frakG} = f_{i_0} \cdot H/\Boxed{\sim_\frakG},
  $$
  where $i_0 \in I$ is chosen so that $w \cdot H = \calO_{i_0} = f_{i_0} \cdot H$.
  Since both $f_{i_0} \cdot G_A \in f_{i_0} \cdot H/\Boxed{\sim_\frakG}$ and
  $f_{i_0} \cdot \alpha \cdot G_A \in f_{i_0} \cdot H/\Boxed{\sim_\frakG}$, the definition
  ensures that $\chi$ assumes both colors on $w \cdot \subobj\frakG AA$. Therefore,
  $|\chi(w \cdot \subobj\frakG AA)| = 2$.
\end{proof}

\begin{PROP}\label{gfrd.prop.small-degs} (cf.~\cite{zucker1,masul-kpt})
  Let $\CC$ be a locally small category such that all the morphisms in $\CC$ are mono,
  and let $\frakG = (G_A)_{A \in \Ob(\CC)}$ be a sequence of groups indexed by the objects
  of $\CC$ chosen so that $G_A \le \Aut(A)$ for all $A \in \Ob(\CC)$.
  Then $t(A) = |G_A| \cdot t^\frakG(A)$ for all $A \in \Ob(\CC)$, having in mind Convention~\ref{crt.CONVENTION.Ninf}.
\end{PROP}
\begin{proof}
  Fix an $A \in \Ob(\CC)$.
  Let us start by showing that $t(A) = \infty$ if $|G_A| = \infty$. We will do so by showing that
  $t(A) \ge n$ for every $n \in \NN$. Fix an $n \in \NN$ and $X \subseteq G_A$ such that $|X| = n$.
  Since $t^\frakG(A) \ge 1$ there is a $k \ge 2$ and a $B \in \Ob(\CC)$ such that for every $C \in \Ob(\CC)$ one can find a coloring
  $\chi : \binom CA \to k$ such that for every $w : B \to C$ we have that
  $|\chi(w \cdot \binom BA)| \ge 1$. This is, of course, trivial. We need this argument
  just to ensure the existence of a $B$ such that $A \to B$.
 
  Let $\subobj \frakG AC = \{H_i : i \in I\}$ for some index set $I$.
  For each $i \in I$ choose a representative $h_i \in H_i$. Then $H_i = h_i \cdot G_A$.
  Fix an arbitrary $\xi \in X$ and define $\chi' : \hom(A, C) \to X$ as follows:
  \begin{itemize}
  \item[]
    if $g = h_i \cdot \alpha$ for some $i \in I$ and some $\alpha \in X$ then $\chi'(g) = \alpha$;
  \item[]
    otherwise $\chi'(g) = \xi$.
  \end{itemize}
  Take any $w : B \to C$. Let $f \in \hom(A, B)$ be arbitrary. Then:
  $$
    |\chi'(w \cdot \hom(A, B))| \ge |\chi'(w \cdot f \cdot G_A)|.
  $$
  Clearly, $w \cdot f \cdot G_A = h_i \cdot G_A$ for some $i \in I$, so
  $$
    |\chi'(w \cdot \hom(A, B))| \ge |\chi'(h_i \cdot G_A)| = n.
  $$
  This completes the proof in case $G_A$ is infinite.
  
  Let us now move on to the case when $G_A$ is finite. We shall distinguish two cases: $t^\frakG(A) = \infty$ and $t^\frakG(A) \in \NN$.

  Assume first that $t^\frakG(A) = \infty$ and let us show, as above, that $t(A) = \infty$ by showing that
  $t(A) \ge n$ for every $n \in \NN$. Fix an $n \in \NN$. Since $t^\frakG(A) = \infty$, there is
  a $k \ge 2$ and a $B \in \Ob(\CC)$ such that for every $C \in \Ob(\CC)$ one can find a coloring
  $\chi : \subobj \frakG AC \to k$ such that for every $w : B \to C$ we have that
  $|\chi(w \cdot \subobj \frakG AB)| \ge n$. Then the coloring $\chi' : \hom(A, C) \to k$ defined by
  $$
    \chi'(f) = \chi(f / \Boxed{\sim_\frakG})
  $$
  has the property that $|\chi(w \cdot \hom(A, B))| \ge n$.

  Let, now, $t^\frakG(A) = n$ for some $n \in \NN$ and let us show that $t(A) = n \cdot |G_A|$ (recall that $G_A$ is also finite).
  Take any $k \ge 2$ and any $B \in \Ob(\CC)$. Then there is a $C \in \Ob(\CC)$
  such that $C \overset{\frakG}\longrightarrow (B)^{A}_{2^k, n}$. Let $\chi : \hom(A, C) \to k$ be an
  arbitrary coloring. Construct $\chi' : \subobj \frakG AC \to \calP(k)$ as follows:
  $$
    \chi'(f / \Boxed{\sim_\frakG}) = \{\chi(g) : g \in f \cdot G_A\}.
  $$
  Then by the choice of $C$ there exists a $w : B \to C$ such that
  $|\chi'(w \cdot \subobj\frakG AB)| \le n$. Since $w \cdot (f / \Boxed{\sim_\frakG}) = (w \cdot f) / \Boxed{\sim_\frakG}$ it follows
  that
  $$\textstyle
    w \cdot \subobj \frakG AB = \{ (w \cdot f) / \Boxed{\sim_\frakG} : f \in \hom(A, B) \}.
  $$
  Moreover, the morphisms in $\CC$ are mono, so $|u / \Boxed{\sim_\frakG}| = |G_A|$
  for each morphism $u \in \hom(A, C)$. Therefore,
  $|\chi'(w \cdot \subobj \frakG AB)| \le n$ implies that $|\chi(w \cdot \hom(A, B))| \le n \cdot |G_A|$
  proving that $t(A) \le n \cdot |G_A|$.

  For the other inequality note that $t^\frakG(A) = n$ implies that there is a $k \ge 2$ and a
  $B \in \Ob(\CC)$ such that for every $C \in \Ob(\CC)$ one can find a coloring
  $\chi_C : \subobj \frakG AC \to k$ with the property that for every $w \in \hom(B, C)$ we have that
  $|\chi_C(w \cdot \subobj \frakG AB)| \ge n$. Take any $C \in \Ob(\CC)$ and
  let $\subobj \frakG AC = \{H_i : i \in I\}$ for some index set $I$.
  For each $i \in I$ choose a representative $h_i \in H_i$ so that $H_i = h_i \cdot G_A$.  
  Since all the morphisms in $\CC$ are mono, for each $f \in \hom(A, C)$ there is a unique
  $i \in I$ and a unique $\alpha \in G_A$ such that $f = h_i \cdot \alpha$. Let us denote this
  $\alpha$ by $\alpha(f)$. Consider the following coloring:
  $$
    \xi : \hom(A, C) \to k \times G_A : f \mapsto (\chi_C(f/\Boxed{\sim_\frakG}), \alpha(f))
  $$
  and take any $w \in \hom(B, C)$. Since
  $|\chi_C(w \cdot \subobj \frakG AB)| \ge n$, it easily follows that $|\xi(w \cdot \hom(A, B))| \ge n \cdot |G_A|$
  proving that $t(A) \ge n \cdot |G_A|$. This completes the proof.
\end{proof}

\begin{COR}
  Let $\CC$ be a locally small category such that all the morphisms in $\CC$ are mono,
  and let $\frakG = (G_A)_{A \in \Ob(\CC)}$ and $\frakH = (H_A)_{A \in \Ob(\CC)}$
  be two sequences of groups indexed by the objects
  of $\CC$ chosen so that $G_A, H_A \le \Aut(A)$ for all $A \in \Ob(\CC)$.
  Then $|G_A| \cdot t^\frakG(A) = |H_A| \cdot t^\frakH(A)$ for all $A \in \Ob(\CC)$, having in mind Convention~\ref{crt.CONVENTION.Ninf}.
\end{COR}

\section{Categories enriched over Top for big Ramsey degrees}
\label{gfrd.sec.big-degs}

A category $\CC$ is \emph{enriched over $\Top$} if each homset is a topological space
and the composition $\Boxed\cdot : \hom(B, C) \times \hom(A, B) \to \hom(A, C)$ is continuous
for all $A, B, C \in \Ob(\CC)$. (Note that a category enriched over $\Top$ has to be locally small.)

\begin{EX}\label{gfrd.ex.discrete-enrichment}
  Any locally small category can be thought of as a category enriched over $\Top$ by taking each
  homset to be a discrete space. We shall refer to this as the \emph{discrete enrichment}.

  In particular, every class of first-order structures can be understood as a category whose morphisms are embeddings
  and with the discrete enrichment. This is the intended interpretation whenever a class of first-order structures
  is treated as a category and the morphisms are not specified.
\end{EX}

\begin{EX}\label{crt.ex.enrichment-for-WchRs}
  The category $\WchRs$ can be enriched over $\Top$ in a nontrivial way as follows:
  each homset $\hom_{\WchRs}((A, \Boxed<), (B, \Boxed<))$ inherits the topology from the Tychonoff topology on $B^A$
  with $B$ discrete. Whenever we refer to $\WchRs$ as a category enriched over $\Top$ we have this particular
  enrichment in mind.
\end{EX}

If $\CC$ is a category enriched over $\Top$ then
the \emph{right multiplication $r_u$} and the \emph{left multiplication $\ell_v$} defined by
\begin{align*}
  r_u &: \hom(B, C) \to \hom(A, C) : x \mapsto x \cdot u \text{ and}\\
  \ell_v &: \hom(A, B) \to \hom(A, C) : x \mapsto v \cdot x.
\end{align*}
are continuous for all $A, B, C \in \Ob(\CC)$ and $u \in \hom(A, B)$, $v \in \hom(B, C)$.
Therefore, the multiplication in $\Aut(A)$ is continuous but the inverse need not be, so
$\Aut(A)$ is a paratopological group for every $A \in \Ob(\CC)$.
However, we will be mostly interested in contexts where the enrichment is Hausdorff and $\hom(A, A)$ is finite,
so $\Aut(A)$ (and all its subgroups) will always be finite discrete, and hence compact topological groups (as the inverse is then trivially
continuous).

For all $A, B \in \Ob(\CC)$ and a subgroup $G \le \Aut(A)$ taken with the topology it inherits from $\hom(A, A)$, the action
$G \times \hom(A, B) \to \hom(A, B)$ given by $(\alpha, f) \mapsto f \cdot \alpha$
is continuous and $G$ acts on $\hom(A, B)$ by homeomorphisms.

Let $\frakG = (G_A)_{A \in \Ob(\CC)}$ be a sequence of groups indexed by the objects of $\CC$ chosen so that
$G_A \le \Aut(A)$ for all $A \in \Ob(\CC)$.
Fix $A, B \in \Ob(\CC)$ and let $q_{AB} : \hom(A, B) \to \subobj{\frakG}AB : f \mapsto f / \Boxed{\sim_\frakG}$ be the mapping
that takes each $f \in \hom(A, B)$ to its equivalence class $f / \Boxed{\sim_\frakG}$. Since $\hom(A, B)$ is a topological space,
$\subobj{\frakG}AB$ can be topologized by the quotient topology, that is, the finest topology that makes $q_{AB}$ continuous.
Let us recall that
\begin{itemize}
  \item if $G_A$ is a topological group then the quotient map $q_{AB} : \hom(A, B) \to \subobj{\frakG}AB$ is open; and
  \item if $\hom(A, B)$ is locally compact second-countable Hausdorff and
        $G_A$ is a compact topological group then $\subobj{\frakG}AB$ is locally compact second-countable Hausdorff.
\end{itemize}
In other words, the locally compact second-countable Hausdorff enrichment is stable under factoring by compact topological groups.

For a topological space $X$ let $\Borel(X)$ denote the $\sigma$-algebra of Borel sets in $X$.
Recall that $\Borel(X) \times \Borel(Y) \subseteq \Borel(X \times Y)$ for all
topological spaces $X$ and $Y$.

Let $G$ be a topological group and let $G \times X \to X : (g, x) \mapsto x^g$ be a continuous action of $G$
on a topological space $X$. For $x \in X$ let $x^G = \{x^g : g \in G\}$ denote the orbit of $x$ in this action
and let $X / G = \{x^G : x \in G\}$ denote the orbit set. Let $q : X \to X / G$ be the natural projection
map $x \mapsto x^G$ and let $X / G$ be topologized by the quotient topology.
In this action of $G$ on $X$ a \emph{cross section} is a set $Y \subseteq X$ 
which picks precisely one element from each orbit, that is, such that
$|Y \cap x^G| = 1$ for all $x \in X$. A \emph{Borel cross section} is a cross section $Y \subseteq X$
such that $Y \in \Borel(X)$. Recall that every continuous action of a locally compact second-countable Hausdorff topological group
$G$ on a second-countable Hausdorff space has a Borel cross section~\cite[Theorem 2]{baker-1965}.

\begin{LEM}\label{crt.lem.borel-cross-section}
  Let $\CC$ be a category enriched over $\Top$. Let $A, S \in \Ob(\CC)$ be chosen so that
  $\hom(A, A)$ and $\hom(A, S)$ are locally compact second-countable Hausdorff.
  If $G_A \le \Aut(A)$ is a topological group closed in $\hom(A, A)$
  then the action of $G_A$ on $\hom(A, S)$ given by $(\alpha, f) \mapsto f \cdot \alpha$ has a Borel cross section.
\end{LEM}
\begin{proof}
  Follows from the above remark, having in mind that $G_A$ is
  locally compact second-countable Hausdorff because it inherits these properties from $\hom(A, A)$.
\end{proof}


A mapping $f : X \to Y$ is \emph{finite-to-one} if $f^{-1}(y)$ is finite for all $y \in Y$, and it is
\emph{countable-to-one} if $f^{-1}(y)$ is countable for all $y \in Y$.

\begin{LEM}\label{crt.lem.Borel->Borel}
  Let $\CC$ be a category enriched over $\Top$ and let $\frakG = (G_A)_{A \in \Ob(\CC)}$ be a sequence of groups such that
  $G_A \le \Aut(A)$ for all $A \in \Ob(\CC)$. Let $A, S \in \Ob(\CC)$ be chosen so that $\hom(A, S)$ is locally
  compact second countable Hausdorff and that $\Aut(A)$ is a finite discrete group, and let
  $q_{AS} : \hom(A, S) \to \subobj{\frakG}AS$ be the quotient map $f \mapsto f / \Boxed{\sim_\frakG}$.
  Then for every Borel set $W$ in $\hom(A, S)$ we have that $q_{AS}(W)$ is a Borel set in $\subobj \frakG AS$.
\end{LEM}
\begin{proof}
  Since $G_A$ is a finite discrete group we have that both $\hom(A, S)$ and $\subobj \frakG AS$ are locally
  compact second-countable Hausdorff, and the quotient map $q_{AS}$
  is finite-to-one. The claim now follows from the well-known fact
  that a locally compact second countable Hausdorff space is Polish, and that
  if $f : X \to Y$ a countable-to-one Borel map between Polish spaces $X$ and $Y$ and if $A \subseteq X$ is Borel in $X$
  then $f(A) \subseteq Y$ is Borel in $Y$.
\end{proof}



For a topological space $X$ and an integer $k \ge 2$ a \emph{Borel $k$-coloring of $X$} is any mapping $\chi : X \to k$
such that $\chi^{-1}(i)$ is Borel in $X$ for all~$i \in k$.

Let $\CC$ be a category enriched over $\Top$ and let $\frakG = (G_A)_{A \in \Ob(\CC)}$
be a sequence of groups indexed by the objects of the category chosen so that
$G_A \le \Aut(A)$ for all $A \in \Ob(\CC)$.
For $A, B, C \in \Ob(\CC)$ we write $C \overset\frakG\Borelarrow (B)^{A}_{k, n}$
to denote that for every Borel $k$-coloring $\chi : \subobj \frakG AC \to k$
there is a morphism $w \in \hom_\CC(B, C)$ such that $|\chi(w \cdot \subobj \frakG AB)| \le n$.

For $A, S \in \Ob(\CC)$ let $T^\frakG_\CC(A, S)$ denote the least positive integer $n$ such that
$S \overset\frakG\Borelarrow (S)^{A}_{k, n}$ for all $k \ge 2$, if such an integer exists.
Otherwise put $T^\frakG_\CC(A, S) = \infty$. The number $T^\frakG_\CC(A, S)$ will be referred to as the
\emph{big $\frakG$-Ramsey degree of $A$ in $S$}.

It is easy to see that in the two extreme cases $\frakG$-Ramsey degrees reduce to embedding
Ramsey degrees and structural Ramsey degrees:
$T^\frakA_\CC(A, S)$ is usually referred to as the \emph{big structural Ramsey degree of $A$ in $S$}, and
$T^\frakI_\CC(A, S)$ is usually referred to as the \emph{big embedding Ramsey degree of $A$ in $S$}.
As in the case of small degrees, big dual Ramsey degrees are nothing but ``direct'' Ramsey degrees
in the opposite category. The infrastructure provided by a categories enriched over $\Top$
together with a sequence of distinguished automorphism groups
$\frakG = (G_A)_{A \in \Ob(\CC)}$ suffices for the uniform treatment of all four kinds
of big Ramsey degrees, Table~\ref{gfrd.fig.taxonomy-big}.
Big embedding Ramsey degrees will be denoted by $T_\CC(A, S)$, and we let $T^\partial_\CC(A, S) = T_{\CC^\op}(A, S)$ denote the big dual
Ramsey degrees. We shall omit the index $\CC$ whenever there is no possibility of confusion.

\begin{table}
  \centering
  \begin{tabular}{cccc}
    \cmidrule[\heavyrulewidth]{1-3}
    Big Ramsey  & \multirow{2}{*}{structural} & \multirow{2}{*}{embedding} \\
    degrees &  &  \\
    \cmidrule{1-3}
    ``direct'' & $T^\frakA_\CC(A, S)$ & $T^\frakI_\CC(A, S) = T_\CC(A, S)$ \\
    \cmidrule{1-3}
    dual   & $T^\frakA_{\CC^\op}(A, S)$ & $T^\frakI_{\CC^\op}(A, S) = T^\partial_\CC(A, S)$ \\
    \cmidrule[\heavyrulewidth]{1-3}
  \end{tabular}
  \caption{Big Ramsey degrees are instances of the same phenomenon}
  \label{gfrd.fig.taxonomy-big}
\end{table}

\begin{EX}
  $(a)$
  The Infinite Ramsey Theorem~\ref{gfrd.thm.IRT} can be understood as the first result about big Ramsey degrees since it
  takes the following form in the context we have just introduced:
  In the category $\ChEmb$ with discrete enrichment we have that
  $T(n, \omega) = 1$ for every $n \in \NN$.

  $(b)$
  The Infinite Dual Ramsey Theorem~\ref{gfrd.thm.IDRT} takes the following form in the above setting:
  In the category $\WchRs$ enriched over $\Top$ as in Example~\ref{crt.ex.enrichment-for-WchRs},
  $T^\partial(n, \omega) = 1$ for every $n \in \NN$.
  Note that in this case the enrichment is fundamental for the validity of the theorem.
\end{EX}

\begin{EX}
  It was proved in \cite{masul-drp-algs} that in any nontrivial equationally definable class of algebras over
  a countable algebraic language taken as a category whose morphisms are epimorphisms, 
  every finite algebra has a finite big dual Ramsey degree with respect to Borel colorings in the
  free algebra on $\omega$ generators.
\end{EX}

We have seen that locally compact second-countable Hausdorff enrichment is stable under factoring by compact topological groups,
which enabled us to give a general definition of big Ramsey degrees. Moreover, this enrichment is compatible with the
two typical situations (outlined above) we would like to model. ``Direct'' big Ramsey degrees are usually computed
in categories with discrete enrichment and with respect to all colorings, while dual big Ramsey degrees are usually computed
in categories where the homsets are endowed with the pointwise convergence topology and with respect to Borel colorings.
Both of these contexts fit into the framework of Borel colorings with respect to a locally compact second-countable Hausdorff enrichment.
Let us start with a simple but useful technical statement.

\begin{LEM}\label{crt.lem.easy-borel}
  Let $\CC$ be a category enriched over $\Top$ and let $\frakG = (G_A)_{A \in \Ob(\CC)}$
  be a sequence of groups indexed by the objects of $\CC$ chosen so that $G_A \le \Aut(A)$ for all $A \in \Ob(\CC)$.
  Let $A, B, C, D \in \Ob(\CC)$ be arbitrary and let $k$ and $t$ be positive integers.
  
  $(a)$ If $C \overset\frakG\Borelarrow (B)^A_{k, t}$ and $C \to D$ then $D \overset\frakG\Borelarrow (B)^A_{k, t}$.

  $(b)$ If $C \overset\frakG\Borelarrow (B)^A_{k, t}$ and $D \to B$ then $C \overset\frakG\Borelarrow (D)^A_{k, t}$.
\end{LEM}
\begin{proof}
  $(a)$
  Fix an $f \in \hom(C, D)$. Take any Borel coloring $\chi : \subobj \frakG AD \to k$ and define 
  a Borel coloring $\chi' : \subobj \frakG AC \to k$ by $\chi'(g / \Boxed{\sim_\frakG}) = \chi(f \cdot g / \Boxed{\sim_\frakG})$.
  Then $C \overset\frakG\Borelarrow (B)^A_{k, t}$ yields that there is a $w \in \hom(B, C)$
  such that $|\chi'(w \cdot \subobj\frakG AB)| \le t$. Hence,
  $|\chi(f \cdot w \cdot \subobj\frakG AB)| \le t$.
  
  $(b)$
  Fix an $f \in \hom(D, B)$. Take any Borel coloring $\chi : \subobj\frakG AC \to k$.
  Then $C \overset\frakG\Borelarrow (B)^A_{k, t}$ yields that there is a $w \in \hom(B, C)$
  such that $|\chi(w \cdot \subobj\frakG AB)| \le t$. Since $f \cdot \subobj\frakG AD \subseteq \subobj \frakG AB$ it follows that
  $|\chi(w \cdot f \cdot \subobj \frakG AD)| \le |\chi(w \cdot \subobj\frakG AB)| \le t$.
\end{proof}

\begin{PROP}\label{rdbas.prop.big-infAut} 
  Let $\CC$ be a category enriched over $\Top$ whose morphisms are mono, and let
  $\frakG = (G_A)_{A \in \Ob(\CC)}$ be a sequence of groups indexed by the objects of $\CC$ chosen so that $G_A \le \Aut(A)$ for all $A \in \Ob(\CC)$.
  Assume that $A, S \in \Ob(\CC)$ have the property that
  $\hom(A, A)$ and $\hom(A, S)$ are locally compact second countable Hausdorff and that $G_A$
  is a topological group closed in $\hom(A, A)$. Then, with the Convention~\ref{crt.CONVENTION.Ninf} in mind,
  $T^\frakG(A, S) \ge |G_A|$. In particular, if $G_A$ is infinite then $T^\frakG(A, S) = \infty$.
\end{PROP}
\begin{proof}
  We are going to show that for every finite $X \subseteq G_A$ we have that
  $T^\frakG(A, S) \ge |X|$. This way if $G_A$ is finite we have that
  $T^\frakG(A, S) \ge |G_A|$, and in case $G_A$ is infinite we get
  $T^\frakG(A, S) \ge n$ for all $n \in \NN$.

  Fix an $n \in \NN$ and $X \subseteq G_A$ such that $|X| = n$.
  Let $X = \{\xi_1, \ldots, \xi_n\}$.
  By Lemma~\ref{crt.lem.borel-cross-section} the action of $G_A$ on $\hom(A, S)$
  given by $(\alpha, f) \mapsto f \cdot \alpha$ has a Borel cross section~$Z$.
  The fact that $Z$ is a cross section means that $|Z \cap f \cdot G_A| = 1$ for all $f \in \hom(A, S)$.
  For $\alpha \in G_A$ let
  $$
    r_\alpha : \hom(A, S) \to \hom(A, S) : f \mapsto f \cdot \alpha
  $$
  denote the right translation by $\alpha$. Since $G_A$ acts on $\hom(A, S)$ by homeomorphisms
  each $r_\alpha$ is a homeomorphism and, hence, takes a Borel set onto a Borel set. Let
  $$
    Z_i = r_{\xi_i}(Z), \quad 1 \le i \le n.
  $$
  Each $Z_i$ is a Borel cross section and $i \ne j \Rightarrow Z_i \cap Z_j = \0$.
  Define $\chi : \hom(A, S) \to \{0, 1, \ldots, n\}$ so that
  $$
    \chi(f) = \begin{cases}
      i, & f \in Z_i, 1 \le i \le n, \\
      0, & \text{otherwise.}
    \end{cases}
  $$
  This is clearly a Borel coloring.
  Take any $w \in \hom(S, S)$ and let us show that $|\chi(w \cdot \hom(A, S))| \ge n$.
  Let $f \in \hom(A, S)$ be arbitrary. Clearly, $w \cdot \hom(A, S) \supseteq w \cdot f \cdot G_A$.
  For each cross-section $Z_i$, $1 \le i \le n$, we have that
  $|Z_i \cap  w \cdot f \cdot G_A| = 1$. Since cross sections are pairwise disjoint the
  $Z_i$'s intersect $w \cdot f \cdot G_A$ each in a different point, so
  $|\chi(w \cdot f \cdot G_A)| \ge n$. Finally,
  $|\chi(w \cdot \hom(A, S))| \ge |\chi(w \cdot f \cdot G_A)| \ge n$.
\end{proof}

\begin{PROP}\label{rdbas.prop.big} (cf.~\cite{Zucker-2,masul-rdbas})
  Let $\CC$ be a category enriched over $\Top$ whose morphisms are mono, and let
  $\frakG = (G_A)_{A \in \Ob(\CC)}$ be a sequence of groups indexed by the objects of $\CC$ chosen so that $G_A \le \Aut(A)$ for all $A \in \Ob(\CC)$.
  Let $A, S \in \Ob(\CC)$ be chosen so that $\hom(A, S)$ is locally compact second countable Hausdorff and assume that $G_A$ is a finite discrete group.
  Then, with the Convention~\ref{crt.CONVENTION.Ninf} in mind,
  $T(A, S) = |G_A| \cdot T^\frakG(A, S)$.
\end{PROP}
\begin{proof}
  Let $q_{AS} : \hom(A, S) \to \subobj{\frakG}AS : f \mapsto f / \Boxed{\sim_\frakG}$ be the quotient map.
  Recall that $q_{AS}$ is continuous and open.

  $(\le)$ If $T^\frakG(A, S) = \infty$ the statement is trivially true.
  Assume, therefore, that $T^\frakG(A, S) = n$ for some $n \in \NN$.
  Take any $k \in \NN$. Let $\chi : \hom(A, S) \to k$ be a Borel coloring.
  Construct $\chi' : \subobj \frakG AS \to \calP(k)$ as follows:
  $$
    \chi'(f / \Boxed{\sim_\frakG}) = \{ \chi(g) : g \in f / \Boxed{\sim_\frakG}\}
  $$
  and let us show that this is a Borel coloring. 
  By Lemma~\ref{crt.lem.Borel->Borel}, for every Borel set $W$ in $\hom(A, S)$
  we have that $q_{AB}(W)$ is Borel in $\subobj \frakG AS$.
  Let $X_j = \chi^{-1}(j)$, $j < k$, be the Borel partition of $\hom(A, S)$
  induced by~$\chi$, and let $W_j = q_{AB}(X_j)$, $j < k$. Clearly,
  $\UNION_{j < k} W_j = \subobj \frakG AS$ and all $W_j$'s are Borel in $\subobj \frakG AS$ as we have just argued.
  Now let $C \in \calP(k)$ be a set of colors. Then, by construction of $\chi'$,
  $$
    (\chi')^{-1}(C) = \Big(\bigcap_{i \in C} W_i\Big) \setminus \Big(\bigcup_{j \in k \setminus C} W_j\Big),
  $$
  which is clearly a Borel set.
  Since $T^\frakG(A, S) = n$ we have that $S \overset\frakG\Borelarrow (S)^{A}_{2^k, n}$,
  so there is a $w \in \hom(S, S)$ such that $|\chi'(w \cdot \subobj \frakG AS)| \le n$.
  It is easy to see that
  $|\chi'(w \cdot \subobj \frakG AS)| \le n$ implies $|\chi(w \cdot \hom(A, S))| \le n \cdot |G_A|$,
  proving thus that $T(A, S) \le n \cdot |G_A|$.

  $(\ge)$
  Assume, first, that $T^\frakG(A, S) = \infty$ and let us show that $T(A, S) = \infty$ by showing that
  $T(A, S) \ge n$ for every $n \in \NN$. Fix an $n \in \NN$. Since $T^\frakG(A, S) = \infty$,
  there is a $k \in \NN$ and a Borel coloring
  $\chi : \subobj \frakG AS \to k$ such that for every $w : S \to S$ we have that
  $|\chi(w \cdot \subobj \frakG AS)| \ge n$. Then the coloring $\chi' : \hom(A, S) \to k$ defined by
  $\chi'(f) = \chi(f / \Boxed{\sim_\frakG}) = \chi(q_{AS}(f))$ is clearly a Borel coloring and
  has the property that $|\chi'(w \cdot \hom(A, S))| \ge n$.
  
  Finally, assume that $T^\frakG(A, S) = n \in \NN$.
  Then $T^\frakG(A, S) \ge n$ so there is a $k \in \NN$ and
  a Borel coloring $\chi : \subobj \frakG AS \to k$ with the property that for every $w \in \hom(S, S)$ we have that
  $|\chi(w \cdot \subobj \frakG AS)| \ge n$.
  By Lemma~\ref{crt.lem.borel-cross-section} the action $G_A \times \hom(A, S) \to \hom(A, S) : 
  (\alpha, f) \mapsto f \cdot \alpha$ has a Borel cross section~$Z$.
  For $\alpha \in G_A$ let
  $$
    r_\alpha : \hom(A, S) \to \hom(A, S) : f \mapsto f \cdot \alpha
  $$
  denote the right translation by $\alpha$. Since $G_A$ acts on $\hom(A, S)$ by homeomorphisms
  each $r_\alpha$ is a homeomorphism and, hence, takes a Borel set onto a Borel set. Let
  $G_A = \{\alpha_0, \ldots, \alpha_{m-1}\}$ where $m = |G_A|$ and let
  $Z_i = r_{\alpha_i}(Z)$, $i < m$.
  Each $Z_i$ is a Borel cross section and $i \ne j \Rightarrow Z_i \cap Z_j = \0$.
  Moreover, $\bigcup_{i < m} Z_i = \hom(A, S)$, so $\{Z_i : i < m\}$ is a Borel partition of $\hom(A, S)$.
  Define $p : \hom(A, S) \to m$ so that
  $$
    \chi(f) = i \text{ if and only if } f \in Z_i, \quad i < n,
  $$
  and consider the coloring
  $$
    \chi' : \hom(A, S) \to k \times m : f \mapsto (\chi(f / \Boxed{\sim_\frakG}), p(f)).
  $$
  To see that $\chi'$ is a Borel coloring note that for $i < k$ and $j < m$:
  $$
    (\chi')^{-1}(i, j) = q^{-1}_{AS}(\chi^{-1}(i)) \cap Z_j
  $$
  which is clearly a Borel set. Now, take any $w \in \hom(S, S)$. Since
  $|\chi(w \cdot \subobj \frakG AS)| \ge n$, it easily follows that
  $|\chi'(w \cdot \hom(A, S))| \ge n \cdot m$
  proving that $T(A, S) \ge n \cdot |G_A|$. This completes the proof.
\end{proof}

The objects $C, D \in \Ob(\CC)$ of a category $\CC$ are \emph{hom-equivalent in $\CC$}
if $C \to D$ and $D \to C$.

\begin{COR}
  Let $C, D \in \Ob(\CC)$ be hom-equivalent objects in a category $\CC$ enriched over $\Top$ whose morphisms are mono,
  and let $\frakG = (G_A)_{A \in \Ob(\CC)}$ be a sequence of groups indexed by the objects of $\CC$ chosen so that $G_A \le \Aut(A)$
  for all $A \in \Ob(\CC)$. Let $A \in \Ob(\CC)$ be chosen so that $A \to C$ (and, then also $A \to D$).

  $(a)$ $T(A, C) = T(A, D)$.
  
  $(b)$ If $\hom(A, S)$ is locally compact second countable Hausdorff for every $S \in \Ob(\CC)$
  and $G_A$ is a finite discrete group then $T^\frakG(A, C) = T^\frakG(A, D)$.
\end{COR}
\begin{proof}
  $(a)$
  Take any $A \in \Ob(\CC)$ and let $t = T(A, C) < \infty$.
  Let $k \in \NN$ be an arbitrary integer. Then $C \Borelarrow(C)^A_{k, t}$.
  Since $C$ and $D$ are hom-equivalent we have that $C \to D$ and $D \to C$,
  so by Lemma~\ref{crt.lem.easy-borel} we have that
  $D \Borelarrow(D)^A_{k, t}$. Therefore, $T(A, D) \le t = T(A, C)$.
  By the same argument $T(A, C) \le T(A, D)$.
  Assume, now that $T(A, C) = \infty$. Then $T(A, D) = \infty$, for, otherwise, the argument above
  would force $T(A, C) < \infty$.

  $(b)$ follows from $(a)$ and Proposition~\ref{rdbas.prop.big}.
\end{proof}

\section{Compactness and approximability}
\label{gfrd.sec.compactness}

In this section we prove a general statement which is a common generalization of a
family of statements of the form
``infinite Ramsey property implies the corresponding finite Ramsey property''.
Using the language of Ramsey degrees these statements take the form
  ``small Ramsey degree $\le$ appropriate big Ramsey degree''.

We start with a statement from~\cite{Zucker-2} that in case of classes of structures
small Ramsey degrees are not greater than the corresponding big Ramsey degrees (computed in an
appropriate \Fraisse\ limit). The result follows immediately from a compactness principle
(Theorem~\ref{akpt.lem->thm.ramseyF-part12}) which we now present.

Let $\DD$ be a full subcategory of $\CC$. Then $S \in \Ob(\CC)$ is \emph{universal for $\DD$} if
$D \to S$ for every $D \in \Ob(\DD)$, and it is \emph{weakly locally finite for $\DD$}~\cite{masul-kpt} if
for every $A, B \in \Ob(\DD)$ and every $e \in \hom(A, S)$, $f \in \hom(B, S)$ there exist $D \in \Ob(\DD)$,
$r \in \hom(D, S)$, $p \in \hom(A, D)$ and $q \in \hom(B, D)$ such that $r \cdot p = e$ and $r \cdot q = f$:
\begin{center}
    \begin{tikzcd}
      D \arrow[rr, "r"] & & S & &  \\
      A \arrow[u, "p"] \arrow[urr, "e"' near start] & & B \arrow[ull, "q" near start] \ar[u, "f"']
    \end{tikzcd}
\end{center}
The following compactness result holds for the discrete enrichment
and in case $S$ is universal and weakly locally finite for $\DD$.

\begin{THM}[Compactness for weakly locally finite objects]\label{akpt.lem->thm.ramseyF-part12} (\cite{masul-kpt,masul-rdbas})
  Let $\DD$ be a full subcategory of a locally small category $\CC$ with discrete enrichment.
  Assume that homsets in $\DD$ are finite. Fix an $S \in \Ob(\CC)$ which is universal and weakly locally finite
  for $\DD$. The following are equivalent for all $t \ge 2$ and all $A \in \Ob(\DD)$:
  \begin{enumerate}\def\labelenumi{(\arabic{enumi})}
  \item
    $t_{\DD}(A) \le t$;
  \item
    $S \longrightarrow (B)^A_{k, t}$ for all $k \in \NN$ and $B \in \Ob(\DD)$ such that $A \to B$.
  \end{enumerate}
  Consequently, $t_\DD(A) \le T_\CC(A, S)$ for every $A \in \Ob(\DD)$.
\end{THM}

This variant of compactness suffices to infer that for every \Fraisse\ class $\KK$ of finite relational structures
and every $A \in \KK$ we have that $t(A) \le T(A, S)$ where $S$ is the \Fraisse\ limit of $\KK$.
However, it does not apply to the setting where homsets are enriched by nontrivial topologies
and big Ramsey degrees are computed with respect to Borel colorings. In particular, it does not apply to dual Ramsey degrees.

A specialization of the above statement is motivated by the theory of projective \Fraisse\ limits by
Irwin and Solecki~\cite{irwin-solecki} and handles both ``direct'' and projective \Fraisse\ limits.
Recall that a sequence in a category $\DD$ is any functor $X : \omega \to \DD$ which we choose
to write as $(X_n, x_n^m)_{n \le m \in \omega}$ where $X_n = X(n) \in \Ob(\DD)$
and $x_n^m \in \hom_\DD(X_n, X_m)$ is the image under $X$ of the only morphism $n \to m$.
Assume that $\DD$ sits as a full subcategory in a larger category $\CC$ and assume that
a sequence $(X_n, x_n^m)_{n \le m \in \omega}$ has a colimit $(S, \sigma_n)_{n \in \omega}$ in $\CC$,
where $\sigma_n : X_n \to S$ are the canonical morphisms, $n \in \omega$.
The colimit $(S, \sigma_n)_{n \in \omega}$ is \emph{$\omega$-small} if for every $A \in \Ob(\DD)$
and every $f \in \hom_\CC(A, S)$ there is an $n \in \omega$ and $g_n \in \hom_\DD(A, X_n)$
such that $f = \sigma_n \cdot g_n$:
\begin{center}
  \begin{tikzcd}
    A \arrow[d, "g_n"'] \arrow[dr, "f"] & \\
    X_n \arrow[r, "\sigma_n"'] & S
  \end{tikzcd}
\end{center}

\begin{COR}[Compactness for $\omega$-small colimits]\label{gfrd.cor.omega-small}
  Let $\CC$ be a category whose morphisms are mono and with the Hausdorff enrichment over $\Top$.
  Let $\DD$ be a full subcategory of $\CC$ such that $\hom(A, B)$ is finite for all $A, B \in \Ob(\DD)$.
  Let $(X_n, x_n^m)_{n \le m \in \omega}$ be a sequence in $\DD$ with the colimit $(S, \sigma_n)_{n \in \omega}$ in $\CC$
  which is $\omega$-small and universal for $\DD$. Then the following are equivalent for every $A \in \Ob(\DD)$ and every $t \in \NN$:
  \begin{enumerate}\def\labelenumi{(\arabic{enumi})}
    \item $t_\DD(A) \le t$;
    \item $S \Borelarrow (B)^A_{k,t}$ for all $k \in \NN$ and all $B \in \Ob(\DD)$ such that $A \to B$.
  \end{enumerate}
  Consequently, $t_\DD(A) \le T_\CC(A, S)$ for every $A \in \Ob(\DD)$.
\end{COR}
\begin{proof}
  Note, first, that $\omega$-small colimits are weakly locally finite.
  Note also that $\hom(A, S)$ is a countable Hausdorff space for every $A \in \Ob(\DD)$
  since $(S, \sigma_n)_{n \in \omega}$ is $\omega$-small and the homsets in $\DD$ are finite.
  Therefore, every coloring $\chi : \hom(A, S) \to k$ is trivially a Borel coloring and the argument of \cite[Lemma 3.4]{masul-kpt} applies.
  Finally, Lemma~\ref{crt.lem.easy-borel} and the implication $(2) \Rightarrow (1)$ prove that
  $t_\DD(A) \le T_\CC(A, S)$ for every $A \in \Ob(\DD)$.
\end{proof}

Let us briefly recall the setup for the theory of projective \Fraisse\ limits of Irwin and Solecki~\cite{irwin-solecki}.
For a relational language $L$ a \emph{topological $L$-structure} is a zero-dimensional compact
second countable space endowed with closed relations that interpret symbols from~$L$.
An \emph{epimorphism} from a topological $L$-structure $\calA$ onto a topological $L$-structure $\calB$ is a continuous surjective
function $f : A \to B$ such that for any $R \in L$ and all $y_1, \ldots, y_n \in B$ we have that
$(y_1, \ldots, y_n) \in R^\calB$ if and only if there exist $x_1, \ldots, x_n \in A$ such that $f(x_i) = y_i$, $1 \le i \le n$,
and $(x_1, \ldots, x_n) \in R^\calA$. An \emph{isomorphism} of two topological $L$-structures is a bijective epimorphism.
Since the topology on a topological $L$-structure is compact, each isomorphism is a homeomorphism.
A family $\KK$ of finite discrete topological structures is a \emph{projective family}~\cite{irwin-solecki} if
for all $\calA, \calB \in \KK$ there is a $\calC \in \KK$ such that
$\calA \leftarrow \calC \rightarrow \calB$ (\emph{projective joint embedding property}); and
every cospan $\calA \rightarrow \calC \leftarrow \calB$ in $\KK$ completes to a square
{\def\arraystretch{0.75}$\begin{array}{c@{\,}c@{\,}c@{\,}c@{\,}c} & & \calD & & \\ & \swarrow & & \searrow \\ \calA & \rightarrow & \cal C & \leftarrow & \calB \end{array}$}
(\emph{projective amalgamation}).

Let $\KK$ be a class of topological $L$-structures. A topological $L$-structure $\calF$ is
a \emph{projective \Fraisse\ limit of $\KK$}~\cite{irwin-solecki} if the following conditions hold:
\begin{itemize}
  \item for every $\calA \in \KK$ there is an epimorphism from $\calF$ to $\calA$ (projective universality);
  \item for any finite discrete topological space $X$ and any continuous function $f : \calF \to X$
        there is an $\calA \in \KK$, an epimorphism $e : \calF \to \calA$ and a function
        $f' : A \to X$ such that $f = f' \circ e$ (projective hereditary property); and
  \item for any $\calA \in \KK$ and epimorphisms $e_1 : \calF \to \calA$ and $e_2 : \calF \to \calA$
        there exists an isomorphism $\psi : \calF \to \calF$ such that $e_2 = e_1 \circ \psi$ (projective ultrahomogeneity).
\end{itemize}
Then \cite{irwin-solecki} shows that every countable projective \Fraisse\ family of finite topological $L$-structures
has a projective \Fraisse\ limit which is unique up to isomorphism.

\begin{COR}
  Let $L$ be a relational language,
  let $\KK$ be a countable projective \Fraisse\ family of finite topological $L$-structures
  and let $\calF$ be the projective \Fraisse\ limit of $\KK$. Then $t_\KK^\partial(\calA) \le T^\partial_{\KK \cup \{\calF\}}(\calA, \calF)$ for every $\calA \in \KK$.
\end{COR}
\begin{proof}
  From the construction of $\calF$ given in~\cite{irwin-solecki} it follows immediately that $\calF$
  is dually $\omega$-small for $\KK$. Since projective universality of $\calF$ is postulated by the definition,
  the result now follows from the dual of Corollary~\ref{gfrd.cor.omega-small}.
\end{proof}

However, the relationship between small and big dual Ramsey degrees of finite linear orders
(that is, the fact that the the Infinite Dual Ramsey Theorem implies the Finite Dual Ramsey Theorem)
is explained neither by Theorem~\ref{akpt.lem->thm.ramseyF-part12} nor by Corollary~\ref{gfrd.cor.omega-small}.
Our goal now is to present a common generalization of all such phenomena, both ``direct'' and dual.
To this end let us deconstruct the proof of the following well-known fact:

\begin{figure}
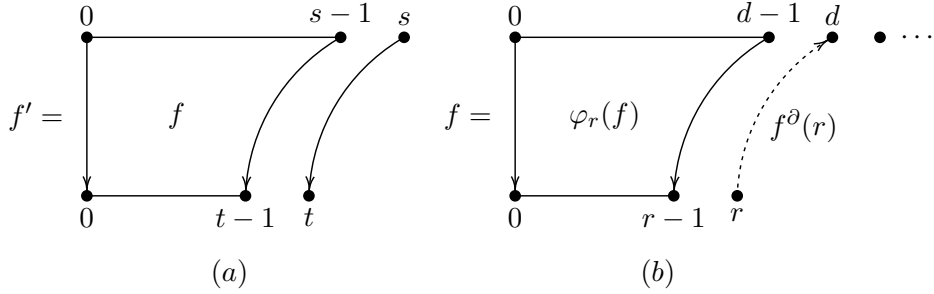

  \centering
\begin{pgfpicture}
  \pgfsetxvec{\pgfpoint{\acadpgfunit}{0pt}}
  \pgfsetyvec{\pgfpoint{0pt}{\acadpgfunit}}
  \pgfsetlinewidth{\acadpgflinewidth}
  \pgftransformshift{\pgfpointxy{62.5}{-75.0}}

  \begin{pgfscope}
    \pgfpathmoveto{\pgfpointxy{75.0}{500.0}}
    \pgfpathlineto{\pgfpointxy{475.0}{500.0}}
    \pgfusepath{stroke}
  \end{pgfscope}
  \begin{pgfscope}
    \pgfpathmoveto{\pgfpointxy{75.0}{250.0}}
    \pgfpathlineto{\pgfpointxy{325.0}{250.0}}
    \pgfusepath{stroke}
  \end{pgfscope}
  \begin{pgfscope}
    \pgfpathmoveto{\pgfpointxy{75.0}{492.0}}
    \pgfpathlineto{\pgfpointxy{75.0}{258.0}}
    \pgfusepath{stroke}
  \end{pgfscope}
  \begin{pgfscope}
    \pgfpathmoveto{\pgfpointxy{81.0289}{280.5}}
    \pgfpatharcaxes{150.0}{180.0}{\pgfpointxy{45.0}{0.0}}{\pgfpointxy{0.0}{45.0}}
    \pgfusepath{stroke}
  \end{pgfscope}
  \begin{pgfscope}
    \pgfpathmoveto{\pgfpointxy{75.0}{258.0}}
    \pgfpatharcaxes{0.0}{30.0}{\pgfpointxy{45.0}{0.0}}{\pgfpointxy{0.0}{45.0}}
    \pgfusepath{stroke}
  \end{pgfscope}
  \begin{pgfscope}
    \pgfpathmoveto{\pgfpointxy{468.126}{495.907}}
    \pgfpatharcaxes{121.533}{176.539}{\pgfpointxy{300.411}{0.0}}{\pgfpointxy{0.0}{300.411}}
    \pgfusepath{stroke}
  \end{pgfscope}
  \begin{pgfscope}
    \pgfpathmoveto{\pgfpointxy{333.335}{279.821}}
    \pgfpatharcaxes{143.402}{176.539}{\pgfpointxy{40.7404}{0.0}}{\pgfpointxy{0.0}{40.7404}}
    \pgfusepath{stroke}
  \end{pgfscope}
  \begin{pgfscope}
    \pgfpathmoveto{\pgfpointxy{325.377}{257.991}}
    \pgfpatharcaxes{-3.46091}{29.6758}{\pgfpointxy{40.7404}{0.0}}{\pgfpointxy{0.0}{40.7404}}
    \pgfusepath{stroke}
  \end{pgfscope}
  \begin{pgfscope}
    \pgfpathmoveto{\pgfpointxy{568.126}{495.907}}
    \pgfpatharcaxes{121.533}{176.539}{\pgfpointxy{300.411}{0.0}}{\pgfpointxy{0.0}{300.411}}
    \pgfusepath{stroke}
  \end{pgfscope}
  \begin{pgfscope}
    \pgfpathmoveto{\pgfpointxy{433.335}{279.821}}
    \pgfpatharcaxes{143.402}{176.539}{\pgfpointxy{40.7404}{0.0}}{\pgfpointxy{0.0}{40.7404}}
    \pgfusepath{stroke}
  \end{pgfscope}
  \begin{pgfscope}
    \pgfpathmoveto{\pgfpointxy{425.377}{257.991}}
    \pgfpatharcaxes{-3.46091}{29.6758}{\pgfpointxy{40.7404}{0.0}}{\pgfpointxy{0.0}{40.7404}}
    \pgfusepath{stroke}
  \end{pgfscope}
  \begin{pgfscope}
    \pgfpathmoveto{\pgfpointxy{750.0}{500.0}}
    \pgfpathlineto{\pgfpointxy{1150.0}{500.0}}
    \pgfusepath{stroke}
  \end{pgfscope}
  \begin{pgfscope}
    \pgfpathmoveto{\pgfpointxy{750.0}{250.0}}
    \pgfpathlineto{\pgfpointxy{1000.0}{250.0}}
    \pgfusepath{stroke}
  \end{pgfscope}
  \begin{pgfscope}
    \pgfpathmoveto{\pgfpointxy{750.0}{492.0}}
    \pgfpathlineto{\pgfpointxy{750.0}{258.0}}
    \pgfusepath{stroke}
  \end{pgfscope}
  \begin{pgfscope}
    \pgfpathmoveto{\pgfpointxy{756.029}{280.5}}
    \pgfpatharcaxes{150.0}{180.0}{\pgfpointxy{45.0}{0.0}}{\pgfpointxy{0.0}{45.0}}
    \pgfusepath{stroke}
  \end{pgfscope}
  \begin{pgfscope}
    \pgfpathmoveto{\pgfpointxy{750.0}{258.0}}
    \pgfpatharcaxes{0.0}{30.0}{\pgfpointxy{45.0}{0.0}}{\pgfpointxy{0.0}{45.0}}
    \pgfusepath{stroke}
  \end{pgfscope}
  \begin{pgfscope}
    \pgfpathmoveto{\pgfpointxy{1143.13}{495.907}}
    \pgfpatharcaxes{121.533}{176.539}{\pgfpointxy{300.411}{0.0}}{\pgfpointxy{0.0}{300.411}}
    \pgfusepath{stroke}
  \end{pgfscope}
  \begin{pgfscope}
    \pgfpathmoveto{\pgfpointxy{1008.33}{279.821}}
    \pgfpatharcaxes{143.402}{176.539}{\pgfpointxy{40.7404}{0.0}}{\pgfpointxy{0.0}{40.7404}}
    \pgfusepath{stroke}
  \end{pgfscope}
  \begin{pgfscope}
    \pgfpathmoveto{\pgfpointxy{1000.38}{257.991}}
    \pgfpatharcaxes{-3.46091}{29.6758}{\pgfpointxy{40.7404}{0.0}}{\pgfpointxy{0.0}{40.7404}}
    \pgfusepath{stroke}
  \end{pgfscope}
  \begin{pgfscope}
    \pgfsetdash{{1.5pt}{2pt}}{0pt}
    \pgfpathmoveto{\pgfpointxy{1243.13}{495.907}}
    \pgfpatharcaxes{121.533}{176.539}{\pgfpointxy{300.411}{0.0}}{\pgfpointxy{0.0}{300.411}}
    \pgfusepath{stroke}
  \end{pgfscope}
  \begin{pgfscope}
    \pgfpathmoveto{\pgfpointxy{1220.68}{489.907}}
    \pgfpatharcaxes{268.397}{301.533}{\pgfpointxy{40.7404}{0.0}}{\pgfpointxy{0.0}{40.7404}}
    \pgfusepath{stroke}
  \end{pgfscope}
  \begin{pgfscope}
    \pgfpathmoveto{\pgfpointxy{1243.13}{495.907}}
    \pgfpatharcaxes{121.533}{154.67}{\pgfpointxy{40.7404}{0.0}}{\pgfpointxy{0.0}{40.7404}}
    \pgfusepath{stroke}
  \end{pgfscope}
  \begin{pgfscope}
    \pgfsetfillcolor{black}
    \pgfpathellipse{\pgfpointxy{75.0}{500.0}}{\pgfpointxy{8.0}{0.0}}{\pgfpointxy{0.0}{8.0}}
    \pgfusepath{fill,stroke}
  \end{pgfscope}
  \begin{pgfscope}
    \pgfsetfillcolor{black}
    \pgfpathellipse{\pgfpointxy{75.0}{250.0}}{\pgfpointxy{8.0}{0.0}}{\pgfpointxy{0.0}{8.0}}
    \pgfusepath{fill,stroke}
  \end{pgfscope}
  \begin{pgfscope}
    \pgfsetfillcolor{black}
    \pgfpathellipse{\pgfpointxy{475.0}{500.0}}{\pgfpointxy{8.0}{0.0}}{\pgfpointxy{0.0}{8.0}}
    \pgfusepath{fill,stroke}
  \end{pgfscope}
  \begin{pgfscope}
    \pgfsetfillcolor{black}
    \pgfpathellipse{\pgfpointxy{325.0}{250.0}}{\pgfpointxy{8.0}{0.0}}{\pgfpointxy{0.0}{8.0}}
    \pgfusepath{fill,stroke}
  \end{pgfscope}
  \begin{pgfscope}
    \pgfsetfillcolor{black}
    \pgfpathellipse{\pgfpointxy{575.0}{500.0}}{\pgfpointxy{8.0}{0.0}}{\pgfpointxy{0.0}{8.0}}
    \pgfusepath{fill,stroke}
  \end{pgfscope}
  \begin{pgfscope}
    \pgfsetfillcolor{black}
    \pgfpathellipse{\pgfpointxy{425.0}{250.0}}{\pgfpointxy{8.0}{0.0}}{\pgfpointxy{0.0}{8.0}}
    \pgfusepath{fill,stroke}
  \end{pgfscope}
  \begin{pgfscope}
    \pgfsetfillcolor{black}
    \pgfpathellipse{\pgfpointxy{750.0}{500.0}}{\pgfpointxy{8.0}{0.0}}{\pgfpointxy{0.0}{8.0}}
    \pgfusepath{fill,stroke}
  \end{pgfscope}
  \begin{pgfscope}
    \pgfsetfillcolor{black}
    \pgfpathellipse{\pgfpointxy{750.0}{250.0}}{\pgfpointxy{8.0}{0.0}}{\pgfpointxy{0.0}{8.0}}
    \pgfusepath{fill,stroke}
  \end{pgfscope}
  \begin{pgfscope}
    \pgfsetfillcolor{black}
    \pgfpathellipse{\pgfpointxy{1150.0}{500.0}}{\pgfpointxy{8.0}{0.0}}{\pgfpointxy{0.0}{8.0}}
    \pgfusepath{fill,stroke}
  \end{pgfscope}
  \begin{pgfscope}
    \pgfsetfillcolor{black}
    \pgfpathellipse{\pgfpointxy{1000.0}{250.0}}{\pgfpointxy{8.0}{0.0}}{\pgfpointxy{0.0}{8.0}}
    \pgfusepath{fill,stroke}
  \end{pgfscope}
  \begin{pgfscope}
    \pgfsetfillcolor{black}
    \pgfpathellipse{\pgfpointxy{1250.0}{500.0}}{\pgfpointxy{8.0}{0.0}}{\pgfpointxy{0.0}{8.0}}
    \pgfusepath{fill,stroke}
  \end{pgfscope}
  \begin{pgfscope}
    \pgfsetfillcolor{black}
    \pgfpathellipse{\pgfpointxy{1100.0}{250.0}}{\pgfpointxy{8.0}{0.0}}{\pgfpointxy{0.0}{8.0}}
    \pgfusepath{fill,stroke}
  \end{pgfscope}
  \begin{pgfscope}
    \pgfsetfillcolor{black}
    \pgfpathellipse{\pgfpointxy{1325.0}{500.0}}{\pgfpointxy{8.0}{0.0}}{\pgfpointxy{0.0}{8.0}}
    \pgfusepath{fill,stroke}
  \end{pgfscope}
  \pgftext[bottom,at={\pgfpointxy{75.0}{520.0}}]{$0$}
  \pgftext[bottom,at={\pgfpointxy{475.0}{520.0}}]{$s-1$}
  \pgftext[bottom,at={\pgfpointxy{575.0}{520.0}}]{$s$}
  \pgftext[top,at={\pgfpointxy{75.0}{230.0}}]{$0$}
  \pgftext[top,at={\pgfpointxy{325.0}{230.0}}]{$t-1$}
  \pgftext[top,at={\pgfpointxy{425.0}{230.0}}]{$t$}
  \pgftext[at={\pgfpointxy{216.017}{376.335}}]{$f$}
  \pgftext[right,at={\pgfpointxy{38.0}{375.0}}]{$f' = $}
  \pgftext[top,at={\pgfpointxy{300.0}{150.5}}]{$(a)$}
  \pgftext[bottom,at={\pgfpointxy{750.0}{520.0}}]{$0$}
  \pgftext[bottom,at={\pgfpointxy{1150.0}{520.0}}]{$d-1$}
  \pgftext[bottom,at={\pgfpointxy{1250.0}{520.0}}]{$d$}
  \pgftext[top,at={\pgfpointxy{750.0}{230.0}}]{$0$}
  \pgftext[top,at={\pgfpointxy{1000.0}{230.0}}]{$r-1$}
  \pgftext[top,at={\pgfpointxy{1100.0}{230.0}}]{$r$}
  \pgftext[at={\pgfpointxy{891.017}{376.335}}]{$\phi_r(f)$}
  \pgftext[right,at={\pgfpointxy{713.0}{375.0}}]{$f = $}
  \pgftext[top,at={\pgfpointxy{975.0}{150.5}}]{$(b)$}
  \pgftext[top,left,at={\pgfpointxy{1150.64}{386.416}}]{$f^\partial(r)$}
  \pgftext[at={\pgfpointxy{1387.5}{500.0}}]{$\cdots$}
\end{pgfpicture}
  \caption{The constructions in the proof that the Infinite Dual Ramsey implies the Finite Dual Ramsey}
  \label{crt.fig.idr01}
\end{figure}

\begin{THM} \cite{carlson-simpson-1984} \label{gfrd.thm.IDRT=>FDRT}
  The Infinite Dual Ramsey Theorem implies the Finite Dual Ramsey Theorem.
\end{THM}
\begin{proof} (Sketch)
  For $f \in \RSurj(s, t)$ we define $f' \in \RSurj(s+1, t+1)$ by $f'(j) = f(j)$ for $j < s$, and $f'(s) = t$,
  Fig.~\ref{crt.fig.idr01}~$(a)$. Clearly, if $f$ is a rigid surjection then so is $f'$.
  Next, define $\phi_r : \RSurj(\omega, r+1) \to \bigcup_{s \ge r} \RSurj(s, r)$
  as follows: $\phi_r(f) = \restr{f}{d}$ where $d = f^\partial(r) = \min f^{-1}(r)$, Fig.~\ref{crt.fig.idr01}~$(b)$.
  Finally, let $\pi_m \in \RSurj(\omega, m)$ be the ``canonical rigid surjection'' defined by
  $\pi_m(j) = j$ for $j < m$ and $\pi_m(j) = m - 1$ for $j \ge m$.

  Seeking a contradiction, assume that there are $r, m, k \in \NN$ with the following property:
  $m \ge r$ and for every $n \ge m$ there is a coloring $\chi_n : \RSurj(n, r) \to k$
  such that for every $w \in \RSurj(n, m)$ we have that
  $$
    |\chi_n(\RSurj(m, r) \circ w)| > 1.
  $$
  Define $\gamma : \RSurj(\omega, r+1) \to k$ by
  $$
    \gamma(f) = \chi_d(\phi_r(f)), \text{ where } d = f^\partial(r) = \min f^{-1}(r).
  $$
  This is clearly a Borel coloring so by the Infinite Dual Ramsey Theorem~\ref{gfrd.thm.IDRT}
  there is a color $c \in k$ and an $h \in \RSurj(\omega, \omega)$ such that
  $$
    \gamma(\RSurj(\omega, r+1) \circ h) = \{c\}.
  $$
  Let $u = \pi_{m+1} \circ h$, let $n = u^\partial(s) = \min u^{-1}(s)$ and
  let us show that $\chi_n(\RSurj(m, r) \circ \phi_m(u)) = \{c\}$. Take any $f \in \RSurj(m, r)$.
  The key insight now is that $f \circ \phi_m(u) = \phi_r(f' \circ u)$:
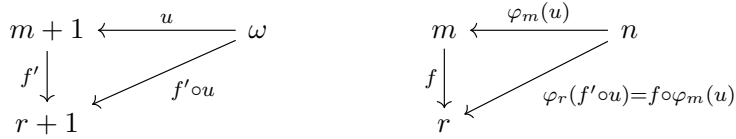
  \begin{center}
    \begin{tikzcd}
      m+1 \arrow[d, "f'"'] & & \omega \arrow[ll, "u"'] \arrow[dll, "f' \circ u"] & & m \arrow[d,"f"'] & & n \arrow[ll, "\phi_m(u)"'] \arrow[dll, "\phi_r(f' \circ u) = f \circ \phi_m(u)"] \\
      r+1                  & &                                                   & & r                & &
    \end{tikzcd}
  \end{center}
  To see that this is indeed the case note that $(f' \circ u)^\partial(r) = u^\partial((f')^\partial(r)) = u^\partial(m) = n$
  and that $\restr{(f' \circ u)}n = f \circ (\restr un)$.
  Finally, $\chi_n(f \circ \phi_m(u)) = \chi_n(\phi_r(f' \circ u)) = \gamma(f' \circ \pi_{m+1} \circ h) = \{c\}$
  because $f' \circ \pi_{m+1} \in \RSurj(\omega, r+1)$. Therefore, 
  $|\chi_n(\RSurj(m, r) \circ \phi_m(u))| = 1$, which contradicts the choice of $\chi_n$.
\end{proof}

Abstracting the above idea leads to the following notion.
Let $\CC$ be a category enriched over $\Top$, let $\DD$ be a full subcategory of $\CC$,
and let $S \in \Ob(\CC)$ be an object in $\CC$. We say that \emph{$S$ is approximable in $\DD$} if there exists a mapping
$F : \Ob(\DD) \to \Ob(\DD)$ and a family of Borel maps indexed by $A \in \Ob(\DD)$:
$$
  \Phi_A : \hom_\CC(F(A), S) \to \bigcup_{C \in \Ob(\DD)} \hom_\DD(A, C)
$$
such that for every $A, B \in \Ob(\DD)$, every $f \in \hom_\DD(A, B)$ and every
$u \in \hom_\CC(F(B), S)$ there is an $f' \in \hom_\DD(F(A), F(B))$ such that
$
  \Phi_A(u \cdot f') = \Phi_B(u) \cdot f
$.
(Note that it is implicit in this requirement that $\Phi_B(u)$ and $\Phi_A(u \cdot f')$ have the same codomain.)
\begin{center}
  \begin{tikzcd}[column sep=large]
    F(B)\arrow[r, "u"]                           & S & & B\arrow[r, "\Phi_B(u)"] & C \\
    F(A)\arrow[u, "f'"]\arrow[ur, "u \cdot f'"'] &   & & A\arrow[u, "f"]\arrow[ur, "\Phi_A(u \cdot f')"'] &   
  \end{tikzcd}
\end{center}
We then say that $S$ is approximable in $\DD$ via $F : \Ob(\DD) \to \Ob(\DD)$ and $(\Phi_A)_{A \in \Ob(\DD)}$.

    Note, first, that $F$ is not required to be a functor: this is just a mapping whose domain and codomain may be proper classes.
    Next, note that the codomain of each $\Phi_A$ may be a proper class of morphisms. Nevertheless, it is a disjoint union of topological spaces
    due to the standard assumption that distinct hom-sets in a category are always disjoint.
    So, we interpret the requirement that $\Phi_A$ be Borel as follows: for every $C \in \Ob(\DD)$ and every open $O \subseteq \hom_\DD(A, C)$
    the inverse image $\Phi_A^{-1}(O)$ is Borel in $\hom_\CC(F(A), S)$.

\begin{THM}[Compactness for approximable objects]\label{gfrd.thm.COMPACTNESS-0}
  Let $\CC$ be a category whose morphisms are mono and with the Hausdorff enrichment over $\Top$,
  let $\DD$ be a full subcategory of $\CC$ which is a skeletal category of finite objects,
  and let $S \in \Ob(\CC)$ be an object universal for $\DD$ and approximable in $\DD$ via $F : \Ob(\DD) \to \Ob(\DD)$
  and $(\Phi_A)_{A \in \Ob(\DD)}$. Fix an $A \in \Ob(\DD)$ and $t \in \NN$.

  $(a)$ If $t_\DD(A) \le t$ then $S \Borelarrow (B)^A_{k,t}$ for all $B \in \Ob(\DD)$ and $k \in \NN$.
  
  $(b)$ If $S \Borelarrow (F(B))^{F(A)}_{k,t}$ for all $B \in \Ob(\DD)$ and $k \in \NN$, then $t_\DD(A) \le t$.

  \noindent
  In particular, $t_\DD(A) \le T_\CC(F(A), S)$ for every $A \in \Ob(\DD)$.
\end{THM}
\begin{proof}
  Fix an $A \in \Ob(\DD)$ and $t \in \NN$.

  $(a)$
  Let $k \in \NN$ and $B \in \Ob(\DD)$ be arbitrary. Since $t_\DD(A) \le t$ there is a $C \in \Ob(\DD)$ such that
  $C \longrightarrow (B)^A_{k,t}$. All the hom-sets in $\DD$ are finite discrete spaces, so
  $C \Borelarrow (B)^A_{k,t}$ trivially. The conclusion now follows by Lemma~\ref{crt.lem.easy-borel} because $S$ is universal for $\DD$.

  $(b)$
  By way of contradiction, suppose that $S \Borelarrow (F(B))^{F(A)}_{k,t}$ for all $B \in \Ob(\DD)$ and $k \in \NN$,
  but $t_\DD(A) > t$. Since $t_\DD(A) > t$, there exist $k \in \NN$ and $B \in \Ob(\DD)$ such that for every $C \in \Ob(\DD)$
  one can find a coloring $\chi_C : \hom_\DD(A, C) \to k$ satisfying
  \begin{equation}\label{gfrd.eq.main-proof-1}
    |\chi_C(w \cdot \hom_\DD(A, B))| > t \text{ for every } w \in \hom_\DD(B, C).
  \end{equation}
  Define $\gamma : \hom_\CC(F(A), S) \to k$ by
  $$
    \gamma(h) = \chi_C(\Phi_A(h)), \text{ where $C$ is the codomain of $\Phi_A(h)$}.
  $$
  Note that this is a Borel coloring because hom-sets in $\DD$ are finite discrete spaces, $\DD$ has countably many objects and
  $(\Phi_A)_{A \in \Ob(\DD)}$ is a family of Borel maps. Since $S \Borelarrow (F(B))^{F(A)}_{k,t}$,
  there is a $w \in \hom_\CC(F(B), S)$ such that
  $$
    |\gamma(w \cdot \hom_\CC(F(A), F(B)))| \le t.
  $$
  Let $C$ be the codomain of $\Phi_B(w)$. Take any $f \in \hom_\DD(A, B)$. By the approximability of $S$ there is
  an $f' \in \hom_\CC(F(A), F(B)) = \hom_\DD(F(A), F(B))$ such that $\Phi_A(w \cdot f') = \Phi_B(w) \cdot f$.
  Then $\chi_C(\Phi_B(w) \cdot f) = \chi_C(\Phi_A(w \cdot f')) = \gamma(w \cdot f') \in \gamma(w \cdot \hom_\CC(F(A), F(B)))$.
  Therefore,
  $$
    \chi_C(\Phi_B(w) \cdot \hom_\DD(A, B)) \subseteq \gamma(w \cdot \hom_\CC(F(A), F(B))),
  $$
  whence $|\chi_C(\Phi_B(w) \cdot \hom_\DD(A, B))| \le t$. Contradiction with~\eqref{gfrd.eq.main-proof-1}.
  This concludes the proof of~$(b)$.

  \medskip

  To show that $t_\DD(A) \le T_\CC(F(A), S)$
  let $T_\CC(F(A), S) = t \in \NN$. Then $S \Borelarrow (S)^{F(A)}_{k, t}$.
  Take any $B \in \Ob(\DD)$. Since $F(B) \to S$
  (because $S$ is universal for $\DD$) it follows that $S \Borelarrow (F(B))^{F(A)}_{k, t}$ by Lemma~\ref{crt.lem.easy-borel}.
  The claim now follows from $(b)$.
\end{proof}

\begin{COR}\label{gfrd.thm.COMPACTNESS}
  Let $\CC$ be a category whose morphisms are mono and with the Hausdorff enrichment over $\Top$,
  let $\DD$ be a full subcategory of $\CC$ which is a skeletal category of finite objects,
  and let $S \in \Ob(\CC)$ be an object universal for $\DD$ and approximable in $\DD$ via $F : \Ob(\DD) \to \Ob(\DD)$
  and $(\Phi_A)_{A \in \Ob(\DD)}$. Let $\frakG = (G_A)_{A \in \Ob(\CC)}$ be a sequence of groups indexed by the objects of the category chosen so that
  $G_A \le \Aut(A)$ for all $A \in \Ob(\CC)$. Assume that $\hom_\CC(A, S)$ is locally compact second countable
  Hausdorff for every $A \in \Ob(\DD)$. Then  $|G_A|$ and $|G_{F(A)}|$ are integers, and
  $$
    t^\frakG_\DD(A) \le \frac{|G_{F(A)}|}{|G_A|} \cdot T^\frakG_\CC(F(A), S),
  $$
  for every $A \in \Ob(\DD)$.
  In particular, if $T^\frakG_\CC(A, S) < \infty$ for every $A \in \Ob(\DD)$ then $t^\frakG_\DD(A) < \infty$ for every $A \in \Ob(\DD)$.
\end{COR}
\begin{proof}
  Note that $G_A$ is a finite discrete group for every $A \in \Ob(\DD)$ because
  $\DD$ is a category of finite objects and the enrichment is Hausdorff. Then by
  Theorem~\ref{gfrd.thm.COMPACTNESS-0} and
  Propositions~\ref{gfrd.prop.small-degs} and~\ref{rdbas.prop.big} we have that
  $$
    |G_A| \cdot t^\frakG_\DD(A) = t_\DD(A) \le T_\CC(F(A), S) = |G_{F(A)}| \cdot T^\frakG_\CC(F(A), S)
  $$
  which proves the corollary.
\end{proof}

\begin{EX}\label{gfrd.ex.approximability-of-omega}
  Let $\CC = \ChEmb$ with discrete enrichment and let $\DD = \ChEmb^\fin$. Let us show that $\omega \in \Ob(\CC)$ is approximable in $\DD$.
  For a finite linear order $A = \{a_1 < a_2 < \ldots < a_n\}$ let $F(A) = A \union \{A\}$ linearly ordered so that $a_i < A$ for all~$i$.
  For an embedding $f : F(A) \hookrightarrow \omega$ let $m = f(A)$ and then define $\Phi_A(f) : A \hookrightarrow \{0, 1, \ldots, m - 1\}$
  so that $\Phi_A(f)(a_i) = f(a_i)$ for all~$i$. Finally, for an embedding $f : A \hookrightarrow B$ define $f' : F(A) \hookrightarrow F(B)$
  to be the embedding
  $$
    f'(x) = \begin{cases}
      f(x), & x \in A\\
      B,    & x = A.
    \end{cases}
  $$
  Then it is easy to check that $\omega$ is approximable in $\DD$ via $F$ and $(\Phi_A)_{A \in \Ob(\DD)}$.
  Theorem~\ref{gfrd.thm.COMPACTNESS} specialized to this context reduces to the fact that the Infinite Ramsey Theorem implies the
  Finite Ramsey Theorem.
\end{EX}

\begin{EX}
  Let $\CC = \WchRs^\op$ enriched over $\Top$ as in Example~\ref{crt.ex.enrichment-for-WchRs} and let $\DD = (\WchRs^\fin)^\op$.
  Let us show that $\omega \in \Ob(\CC)$ is approximable in~$\DD$.
  As above, for a finite linear order $A = \{a_1 < a_2 < \ldots < a_n\}$ let $F(A) = A \union \{A\}$ linearly ordered so that $a_i < A$ for all~$i$.
  For a rigid surjection $f : \omega \twoheadrightarrow F(A)$
  let $m = \min f^{-1}(A)$ and then define $\Phi_A(f) : \{0, 1, \ldots, m - 1\} \twoheadrightarrow A$
  so that $\Phi_A(f)(j) = f(j)$ for all~$j < m$.
  Finally, for a rigid surjection $f : B \twoheadrightarrow A$ define $f' : F(B) \twoheadrightarrow F(A)$
  to be the rigid surjection
  $$
    f'(x) = \begin{cases}
      f(x), & x \in B\\
      A,    & x = B.
    \end{cases}
  $$
  Then it is easy to check that $\omega$ is approximable in $\DD$ via $F$ and $(\Phi_A)_{A \in \Ob(\DD)}$.
  Theorem~\ref{gfrd.thm.COMPACTNESS} specialized to this context reduces to the fact that the Infinite Dual Ramsey Theorem implies the
  Finite Dual Ramsey Theorem.
\end{EX}

\section{The $\star$-Ramsey property}
\label{gfrd.sec.star-ramsey}

In the early 1980's Voigt proved an infinitary version of the Finite Dual Ramsey Theorem
(see~\cite[Theorem~A]{promel-voigt-1985}), usually referred to as the infinitary $\star$-version of the Graham-Rothschild theorem
because it involves the partial substitution $\star$ of parameter words.
Its convenience in purely combinatorial applications stems from the fact that it is an infinitary statement with no topological requirements.

\begin{THM}[The infinitary $\star$-version of Graham-Rothschild] \cite[Theorem~A]{promel-voigt-1985}\label{gfrd.thm.CS-star-version}
  Let $A$ be a finite linearly ordered alphabet and $r, k \in \NN$. For every coloring
  $\chi : \bigcup_{r \le s \le \omega} W^s_r(A) \to k$ there is a $w \in W^\omega_\omega(A)$ such that
  $\bigcup_{r \le s \le \omega} w \star W^s_r(A)$ is monochromatic with respect to~$\chi$.
  (Here, $\star$ denotes the partial substitution of parameter words introduced in Section~\ref{gfrd.sec.prelim}.)
\end{THM}

Let us sketch the proof of a special case of this statement, but in the parlance of rigid surjections.
Using the notions introduced in the proof of Theorem~\ref{gfrd.thm.IDRT=>FDRT},
for $h \in \RSurj(\omega, \omega)$ and $f \in \RSurj(s, r)$ where $s \ge r$ let
$$
  h \ostar f = f \circ \phi_s(\pi_{s+1} \circ h).
$$

\begin{THM}
  Take any $r, k \in \NN$. For every coloring $\chi : \bigcup_{r \le s < \omega} \RSurj(s, r) \to k$
  there is a $h \in \RSurj(\omega, \omega)$ such that
  $
    \big|\chi\big(\bigcup_{r \le s < \omega} h \ostar \RSurj(s, r) \big)\big| = 1
  $.
\end{THM}
\begin{proof} (Sketch)
  The proof is very similar to the proof of Theorem~\ref{gfrd.thm.IDRT=>FDRT}.
  Each coloring $\chi : \bigcup_{s \ge r} \RSurj(s, r) \to k$ can be understood as a family of colorings
  $\chi_n : \RSurj(n, r) \to k$, $n \ge r$. So, we define $\gamma : \RSurj(\omega, r+1) \to k$ by
  $$
    \gamma(f) = \chi_d(\phi_r(f)), \text{ where } d = f^\partial(r) = \min f^{-1}(r).
  $$
  This is clearly a Borel coloring so by the Infinite Dual Ramsey Theorem~\ref{gfrd.thm.IDRT}
  there is a color $c \in k$ and an $h \in \RSurj(\omega, \omega)$ such that $\gamma(\RSurj(\omega, r+1) \circ h) = \{c\}$.
  Let $s \ge r$ be arbitrary. Put $u = \pi_{s+1} \circ h$ and $n = u^\partial(s) = \min u^{-1}(s)$.
  In full analogy with the proof of Theorem~\ref{gfrd.thm.IDRT=>FDRT} we
  show that $\chi_n(\RSurj(s, r) \circ \phi_s(u)) = \{c\}$ since, for any $f \in \RSurj(s, r)$,
  we have that that $f \circ \phi_s(u) = \phi_r(f' \circ u)$:
  \begin{center}
    \begin{tikzcd}
      s+1 \arrow[d, "f'"'] & & \omega \arrow[ll, "u"'] \arrow[dll, "f' \circ u"] & & s \arrow[d,"f"'] & & n \arrow[ll, "\phi_s(u)"'] \arrow[dll, "\phi_r(f' \circ u)"] \\
      r+1                  & &                                                   & & r                & &
    \end{tikzcd}
  \end{center}
  Therefore, 
  $\chi_n(h \ostar \RSurj(s, r)) = \chi_n(\RSurj(s, r) \circ \phi_s(\pi_{s+1} \circ h)) = \{c\}$ for every $s \ge r$.
\end{proof}

This proof motivates the following general notion of $\star$-composition.
Let $\CC$ be a category enriched over $\Top$, let $\DD$ be a full subcategory of $\CC$,
and assume that $S \in \Ob(\CC)$ is universal for $\DD$ and approximable in $\DD$ via $F$ and $(\Phi_A)_{A \in \Ob(\DD)}$.
For each $B \in \Ob(\DD)$ fix a morphism $\iota_{B} \in \hom(B, S)$.
Then, in full analogy with the discussion above we can define the $\star$
operation (with respect to $F$, $(\Phi_A)_{A \in \Ob(\DD)}$ and $(\iota_{B})_{B \in \Ob(\DD)}$) as follows:
for an $h \in \hom(S, S)$ and $f \in \hom(A, B)$ where $A, B \in \Ob(\DD)$ let
\begin{equation}\label{gfrd.eq.star-def}
  h \star f = \Phi_B(h \cdot \iota_{F(B)}) \cdot f \in \Mor(\DD).
\end{equation}

\begin{THM}[The $\star$-Ramsey property]\label{gfrd.thm.struct-enum-star-prop}
  Let $\CC$ be a category whose morphisms are mono and with the Hausdorff enrichment over $\Top$,
  and let $\DD$ be a full subcategory of $\CC$ which is a skeletal category of finite objects.
  Let $S \in \Ob(\CC)$ be an object universal for $\DD$ and for each $B \in \Ob(\DD)$ fix a morphism $\iota_{B} \in \hom(B, S)$.
  Assume that $T(A, S) < \infty$ for all $A \in \Ob(\DD)$ and that $S$ is approximable in $\DD$ via $F : \Ob(\DD) \to \Ob(\DD)$
  and $(\Phi_A)_{A \in \Ob(\DD)}$.  
  Then for every $A \in \Ob(\DD)$ there is an $n \in \NN$ such that for every coloring
  $
    \chi : \bigcup_{B \in \Ob(\DD)} \hom(A, B) \to k
  $
  there is an $h \in \hom(S, S)$ satisfying
  $
    \big|\chi \big( \bigcup_{B \in \Ob(\DD)} h \star \hom(A, B) \big) \big| \le n
  $.
  Here, $\star$ is defined by \eqref{gfrd.eq.star-def} with respect to $F$, $(\Phi_{A})_{A \in \Ob(\DD)}$ and $(\iota_{B})_{B \in \Ob(\DD)}$.
\end{THM}
\begin{proof}
  Take any $A \in \Ob(\DD)$ and let $n = T(F(A), S)$, which is an integer by the assumption. Let
  $
    \chi : \bigcup_{B \in \Ob(\DD)} \hom(A, B) \to k
  $
  be a coloring. The coloring $\chi$ can be understood as a family of colorings $\chi_{B} : \hom(A, B) \to k$, $B \in \Ob(\DD)$.
  Define $\gamma : \hom(F(A), S) \to k$ by
  $$
    \gamma(f) = \chi_{C}(\Phi_{A}(f)), \text{ where $C = \cod(\Phi_{A}(f))$}.
  $$
  As in the proof of Theorem~\ref{gfrd.thm.COMPACTNESS-0} it easily follows that $\gamma$ is a Borel coloring.
  Then by the choice of $n$ there is an $h \in \hom(S, S)$ such that
  $$
    |\gamma(h \cdot \hom(F(A), S))| \le n.
  $$
  Take any $B \in \Ob(\DD)$, let $C$ be the codomain of $\Phi_{B}(h \cdot \iota_{F(B)})$ and let us show that
  $$
    \chi_{C}(\Phi_{B}(h \cdot \iota_{F(B)}) \cdot \hom(A, B)) \subseteq \gamma(h \cdot \hom(F(A), S)).
  $$
  Take any $f \in \hom(A, B)$. By the approximability of $S$ there is
  an $f' \in \hom(F(A), F(B))$ such that $\Phi_{A}(h \cdot \iota_{F(B)} \cdot f') = \Phi_{B}(h \cdot \iota_{F(B)}) \cdot f$.
  Then $\chi_{C}(\Phi_{B}(h \cdot \iota_{F(B)}) \cdot f) = \chi_{C}(\Phi_{A}(h \cdot \iota_{F(B)} \cdot f'))
  = \gamma(h \cdot \iota_{F(B)} \cdot f') \in \gamma(h \cdot \hom(F(A), S))$.
  This completes the proof.
\end{proof}

\begin{COR}
  Let $\KK = \LO^\fin$ be the class of all finite linear orders.
  For every $\calA \in \KK$ and every coloring
  $
    \chi : \bigcup_{\calB \in \KK} \Emb(\calA, \calB) \to k
  $
  there is an $h \in \Emb(\omega, \omega)$ such that
  $
    \big|\chi \big( \bigcup_{\calB \in \KK} h \star \Emb(\calA, \calB) \big) \big| = 1
  $,
  where $\star$ is defined with respect to $F$ and $\Phi_A$ from Example~\ref{gfrd.ex.approximability-of-omega}
  and an arbitrary but fixed family of embeddings $\iota_B : B \hookrightarrow \omega$, $B \in \KK$.
\end{COR}
\begin{proof}
  Analogous to the proof of Theorem~\ref{gfrd.thm.struct-enum-star-prop}; just note that 
  $T(F(A), \omega) = 1$ for every $A \in \KK$ by the Infinite Ramsey Theorem.
\end{proof}

Let us now derive the $\star$-Ramsey property for \Fraisse\ limits with strongly amalgamable ages.
For convenience, instead of working with the entire age of the structure, we shall be working with
the set of all finite substructures of the structure.
Let $L$ be a relational language. For an $L$-structure $\calS = (S, L^S)$ let
$$
  \Fin(\calS) = \{\calS[A] : \0 \ne A \subseteq S \land |A| < \infty \}.
$$
Note that $\Fin(\calS)$ is a set and that $\Age(\calS)$ contains $\Fin(\calS)$. Moreover,
$\Fin(\calS)$ and $\Age(\calS)$ are equivalent as categories.
We say that $\calF$ has \emph{natural 1-point extensions} if there is a functor
$J : \Fin(\calF) \to \Fin(\calF)$ such that
\begin{itemize}
  \item for all $\calA = (A, L^A) \in \Fin(\calF)$, if $J(\calA) = (A', L^{A'})$ then $\calA \le J(\calA)$ and $|A' \setminus A| = 1$
        ($J(\calA)$ is a 1-point extension of $\calA$), and
  \item for all $\calB = (B, L^B) \in \Fin(\calF)$ and $f : \calA \hookrightarrow \calB$ if $J(\calB) = (B', L^{B'})$ then
        $J(f)(a) = f(a)$ for all $a \in A$ and $J(f)(A' \setminus A) = B' \setminus B$ ($J$ is natural).
\end{itemize}

\begin{EX}
  $(a)$ Let $\calF$ be the Rado graph $\calR$, one of the Henson graphs $\calH_n$, or the random digraph $\calD$.
  Then $\calF$ has natural 1-point extensions: take $J(\calA)$ to be
  $\calA$ together with a new isolated vertex $v_A$.

  $(b)$ Let $\calF$ be the random poset $\calP$ or the random tournament $\calT$.
  Then $\calF$ has natural 1-point extensions: take $J(\calA)$ to be
  $\calA$ together with a new element $x_A$ which is greater than (arrows) every other element in~$\calA$.

  $(c)$ Let $\calM(\Delta)$ be the random metric space whose distance set $\Delta \subseteq \RR$ has the maximum~$m = \max \Delta \in \Delta$.
  Then $\calM(\Delta)$ has natural 1-point extensions: take $J(\calA)$ to be
  $\calA$ together with a new point $x_A$ placed at distance~$m$ from every other point in~$\calA$.
\end{EX}

Let $L$ be a relational language and $\Boxed{<} \notin L$ a binary relational symbol.
An \emph{enumerated $L$-structure} is an $(L \cup \{\Boxed<\})$-structure $(S, L^S, \Boxed{<^S})$ such that
$(S, \Boxed{<^S})$ is a linear order of order type~$\omega$. A slightly stronger version of
the following property of enumerated countable relational structures was shown in~\cite[Theorem 4.1]{masul-bigrd}:

\begin{THM} (cf.~\cite[Theorem 4.1]{masul-bigrd})\label{gfrd.thm.T41-masul-bigrd}
  Let $\calF$ be a countably infinite relational structure such that $\Fin(\calF)$ has the strong amalgamation
  property. Let $\KK = \Fin(\calF) * \LO$, and let $\sqsubset$ be a linear order on $F$ such that $(F, \Boxed\sqsubset)$
  has order type~$\omega$. Then:
  
  $(a)$ $\Fin(\calF, \Boxed\sqsubset) = \KK$.
  
  $(b)$ For each $(\calA, \Boxed\prec) \in \KK$ we have that $T((\calA, \Boxed\prec), (\calF, \Boxed\sqsubset)) \le
  T(\calA, \calF)$, or, in other words, if $\calA$ has finite big Ramsey degree in $\calF$ then
  $(\calA, \Boxed\prec)$ has finite big Ramsey degree in $(\calF, \Boxed\sqsubset)$.
\end{THM}

The following theorem shows that approximability, despite its unusual definition, is not an artificial phenomenon but a
ubiquitous property that can be found in many natural contexts.

\begin{THM}[The $\star$-Ramsey property for enumerated \Fraisse\ limits]\label{gfrd.thm.struct-star-property}
  Let $L$ be a relational language and $\Boxed{<} \notin L$ a binary relational symbol. Let
  $\calS = (S, L^{S}, \Boxed{<^S})$ be an enumerated $L$-structure where
  $(S, L^{S})$ is a \Fraisse\ limit with big Ramsey degrees and with strongly amalgamable age which has natural 1-point extensions.
  Let $\KK = \Fin(\calS)$. Then:
  
  \medskip

  $(a)$ $\KK$ has natural 1-point extensions.

  \medskip

  $(b)$ $\calS$ is approximable in $\KK$.

  \medskip

  $(c)$ Fix a functor $F : \KK \to \KK$ and a family of morphisms $(\Phi_{\calA})_{\calA \in \KK}$
  which demonstrate the approximability of $\calS$ in $\KK$,
  and for each $\calB \in \KK$ fix an embedding $\iota_{\calB} : \calB \hookrightarrow \calS$.
  For every $\calA \in \KK$ there is an $n \in \NN$ such that for every coloring
  $
    \chi : \bigcup_{\calB \in \KK} \Emb(\calA, \calB) \to k
  $
  there is an $h \in \Emb(\calS, \calS)$ satisfying
  $
    \big|\chi \big( \bigcup_{\calB \in \KK} h \star \Emb(\calA, \calB) \big) \big| \le n
  $,
  where $\star$ is defined by \eqref{gfrd.eq.star-def} with respect to $F$, $(\Phi_{\calA})_{\calA \in \KK}$ and $(\iota_{\calB})_{\calB \in \KK}$.
\end{THM}
\begin{proof}
  $(a)$ Let $J : \Fin(S, L^S) \to \Fin(S, L^S)$ be a functor showing that $\Fin(S, L^S)$ has natural 1-point extensions.
  Define a functor
  $F : \KK \to \KK$ as follows. Take any $(A, L^A, \Boxed{<^A}) \in \KK$ and let $J(A, L^A) = (A', L^{A'})$.
  Define a linear order $<'$ on $A'$ by expanding the order $<$ (which is defined on $A$) so that the only element in
  $A' \setminus A$ becomes the largest element in $(A', \Boxed{<'})$.
  Now put $F(A, L^A, \Boxed<) = (A', L^{A'}, \Boxed{<'})$. Note that $(A', L^{A'}, \Boxed{<'}) \in \KK$ by Theorem~\ref{gfrd.thm.T41-masul-bigrd}~$(a)$.
  This is how $F$ acts on objects.
    
  To see how $F$ acts on morphisms take any
  $(A, L^A, \Boxed<), (B, L^B, \Boxed<) \in \KK$ and an embedding $f : (A, L^A, \Boxed<) \hookrightarrow (B, L^B, \Boxed<)$.
  Then $f$ is also an embedding $(A, L^A) \hookrightarrow (B, L^B)$, so $J(f)$ is an embedding
  $(A', L^{A'}) \hookrightarrow (B', L^{B'})$. The definitions of $J$ and $<'$ ensure that $F(f) = J(f)$ is an embedding
  $(A', L^{A'}, \Boxed{<'}) \hookrightarrow (B', L^{B'}, \Boxed{<'})$. So, $F : \KK \to \KK$ is a functor demonstrating that $\KK$ has
  natural 1-point extensions.

  \medskip

  $(b)$ To show that $\calS$ is approximable in $\KK$ take $F$ to be the functor constructed in $(a)$.
  To define $(\Phi_{\calA})_{\calA \in \KK}$ fix an $\calA = (A, L^A, \Boxed{<^A}) \in \KK$ and take any
  $f : F(\calA) \hookrightarrow \calS$.
  Let $b_1 < b_2 < \ldots < b_k$ be the smallest initial segment of $(S, \Boxed{<^S})$ which contains $\im(f)$ and
  let $\calC = \calS[b_1, b_2, \ldots, b_{k-1}]$. Then define
  $\Phi_{\calA}(f) : \calA \hookrightarrow \calC$ by
  $$
    \Phi_{\calA}(f) = f|^A_{\{b_1, b_2, \ldots, b_{k-1}\}},
  $$
  where $f|^A_{\{b_1, b_2, \ldots, b_{k-1}\}}$ is the restriction of~$f$ to the domain $A$ and the codomain $\{b_1, b_2, \ldots, b_{k-1}\}$.

  \medskip

  $(c)$
  Since $(S, L^S)$ has strongly amalgamable age and finite big Ramsey degrees,
  Theorem~\ref{gfrd.thm.T41-masul-bigrd} applies and yields $T(\calA, \calS) < \infty$ for all $\calA \in \KK$.
  Now, take any $\calA \in \KK$ and let
  $$
    \chi : \bigcup_{\calB \in \KK} \Emb(\calA, \calB) \to k
  $$
  be a coloring. The coloring $\chi$ can be understood as a family of colorings $\chi_{\calB} : \Emb(\calA, \calB) \to k$, $\calB \in \KK$.
  Define $\gamma : \Emb(F(\calA), \calS) \to k$ by
  $$
    \gamma(f) = \chi_{\calC}(\Phi_{\calA}(f)), \text{ where $\calC$ is the codomain of $\Phi_{\calA}(f)$}.
  $$
  Let $n = T(F(\calA), \calS)$ which, as we have seen, is an integer. Then there is an $h \in \Emb(\calS, \calS)$ such that
  $$
    |\gamma(h \circ \Emb(F(\calA), \calS))| \le n.
  $$
  Let $\calC$ be the codomain of $\Phi_{\calB}(h \circ \iota_{F(\calB)})$ and let us show that
  $$
    \chi_{\calC}(\Phi_{\calB}(h \circ \iota_{F(\calB)}) \circ \Emb(\calA, \calB)) \subseteq \gamma(h \circ \Emb(F(\calA), \calS)).
  $$
  Take any $f \in \Emb(\calA, \calB)$. By the approximability of $\calS$ there is
  an $f' \in \Emb(F(\calA), F(\calB))$ such that $\Phi_{\calA}(h \circ \iota_{F(\calB)} \circ f') = \Phi_{\calB}(h \circ \iota_{F(\calB)}) \circ f$.
  Then $\chi_{\calC}(\Phi_{\calB}(h \circ \iota_{F(\calB)}) \circ f) = \chi_{\calC}(\Phi_{\calA}(h \circ \iota_{F(\calB)} \circ f'))
  = \gamma(h \circ \iota_{F(\calB)} \circ f') \in \gamma(h \circ \Emb(F(\calA), \calS))$.
  This completes the proof of $(c)$ and the proof of the theorem.
\end{proof}


\begin{COR}
  Let $L$ be a relational language and $\Boxed{<} \notin L$ a binary relational symbol. Let
  $\calS = (S, L^{S}, \Boxed{<^S})$ be an enumerated $L$-structure where
  $(S, L^{S})$ is a \Fraisse\ limit with big Ramsey degrees and with strongly amalgamable age which has natural 1-point extensions.
  For every $\calA \in \Fin(\calS)$ there is an $n \in \NN$ such that for every $k \in \NN$ and every
  family of colorings $\gamma_\calB : \binom \calB \calA \to k$ indexed by $\calB \in \Fin(\calS)$
  there is a substructure $\calS'$ of $\calS$ which is
  isomorphic to $\calS$, and a choice function $c : \binom{\calS'}\calA \to \Fin(\calS)$ such that
  $$\textstyle
    \Big|\Big\{\gamma_{c(\calE)}(\calE) : \calE \in \binom{\calS'}\calA \Big\}\Big| \le n.
  $$
\end{COR}
\begin{proof}
  By Theorem~\ref{gfrd.thm.struct-star-property}~$(b)$ we know that $\calS$ is approximable in $\Fin(\calS)$.
  Let $F : \Fin(\calS) \to \Fin(\calS)$ be the functor and $(\Phi_\calA)_{\calA \in \Fin(\calS)}$ the family 
  constructed in Theorem~\ref{gfrd.thm.struct-star-property}~$(a)$ and $(b)$, respectively.
  For each $\calB \in \Fin(\calS)$ let $\iota_\calB : \calB \hookrightarrow \calS$ be the inclusion $\iota_\calB(x) = x$.

  Take any $\calA \in \Fin(\calS)$ and let $n \in \NN$ be the integer provided for $\calA$ by the conclusion of
  Theorem~\ref{gfrd.thm.struct-star-property}~$(c)$. Take any $k \in \NN$ and for each $\calB \in \Fin(\calS)$ fix a coloring
  $\gamma_\calB : \binom \calB \calA \to k$.

  Let us define $\chi : \bigcup_{\calB \in \Fin(\calS)} \Emb(\calA, \calB) \to k$ as follows:
  for $\calB \in \Fin(\calS)$ and $f \in \Emb(\calA, \calB)$ let
  $$
    \chi(f) = \gamma_\calB(\im(f)) = \gamma_{\cod(f)}(\im(f)).
  $$
  By Theorem~\ref{gfrd.thm.struct-star-property}~$(c)$ there is an embedding $h : \calS \hookrightarrow \calS$
  such that
  \begin{equation}\label{gfrd.eq.final-1}
    \textstyle
    \big|\chi \big( \bigcup_{\calB \in \Fin(\calS)} h \star \Emb(\calA, \calB) \big) \big| \le n,
  \end{equation}
  where $\star$ is defined by \eqref{gfrd.eq.star-def} with respect to $F$, $(\Phi_{\calA})_{\calA \in \Fin(\calS)}$
  and $(\iota_{\calB})_{\calB \in \Fin(\calS)}$.

  Let $\calS' = \im(h)$. Let us define $c : \binom{\calS'}\calA \to \Fin(\calS)$ as follows.
  Take any $\calE \in \binom{\calS'}{\calA}$. Since $\calS$ is an enumerated structure, $\calA$ and $\calE$ are finite
  ordered structures, so there is a unique isomorphism $f : \calA \to \calE$. Moreover, there is a unique $\calC \in \binom \calS \calA$
  such that $h(\calC) = \calE$. Let $g : \calA \to \calC$ be the unique isomorphism and note that
  $
    f = h|^\calC_\calE \circ g
  $,
  where $h|^\calC_\calE$ is the restriction of $h$ to the domain $\calC$ and codomain $\calE$. Now, let
  $$
    c(\calE) = \cod(h \star g).
  $$
  Recall that $h \star g = \Phi_\calC(h \circ \iota_{F(\calC)}) \circ g$ whence $\im(h \star g) = \calE$ by the
  choice of $\iota_{F(\calC)}$ and $\Phi_\calC$. Therefore,
  $$\textstyle
    \gamma_{c(\calE)}(\calE) = \gamma_{\cod(h \star g)}(\im(h \star g)) = \chi(h \star g) \in \chi \big( \bigcup_{\calB \in \Fin(\calS)} h \star \Emb(\calA, \calB) \big).
  $$
  Consequently,
  $$\textstyle
    \Big\{\gamma_{c(\calE)}(\calE) : \calE \in \binom{\calS'}\calA \Big\}
    \subseteq
    \chi \big( \bigcup_{\calB \in \Fin(\calS)} h \star \Emb(\calA, \calB) \big),
  $$
  which together with \eqref{gfrd.eq.final-1} concludes the proof.
\end{proof}

\section*{Acknowledgements}

The author would like to express his gratitude to J.~Hubi\v cka for many useful comments about big Ramsey degrees of metric spaces.

This research was supported by the Science Fund of the Republic of Serbia, Grant No.~7750027:
Set-theoretic, model-theoretic and Ramsey-theoretic phenomena in mathematical structures: similarity and diversity -- SMART.


\end{document}